\documentclass[12pt]{article}
\usepackage{amsmath,amsfonts,amssymb,amsthm,accents,diagbox}
\usepackage{calrsfs,color,cases,graphicx,mathrsfs,indentfirst,enumitem}
\usepackage{ragged2e}
\usepackage{multirow}
\usepackage{hyperref}
\usepackage[margin=1in]{geometry}
\usepackage[all]{xy}

\newtheorem{definition}{Definition}
\newtheorem{theorem}{Theorem}
\newtheorem{lemma}{Lemma}

\newtheorem{corollary}{Corollary}
\newtheorem{remark}{Remark}

\newtheorem{counterexample}{Counterexample}
\makeatletter

\newcommand{\Rmnum}[1]{\expandafter\@slowromancap\romannumeral #1@}
\makeatother

\begin{document}

\title{\bf Asymptotic theory of $C$-pseudo-cones\footnote{This paper was supported by the NSFC (No. 12371060), the Shaanxi Fundamental Science Research Project for Mathematics and Physics (No. 22JSZ012) and the Excellent Graduate Training Program of SNNU (No. LHRCCX23142).}}
\author{Xudong Wang, Wenxue Xu, Jiazu Zhou, Baocheng Zhu}

\date{}
\maketitle

\begin{abstract}

In this paper, we study the non-degenerated $C$-pseudo-cones which can be uniquely decomposed into the sum of a $C$-asymptotic set and a $C$-starting point. Combining this with the novel work in \cite{Schneider-A_weighted_Minkowski_theorem}, we introduce the asymptotic weighted co-volume functional $T_\Theta(E)$ of the non-degenerated $C$-pseudo-cone $E$, which is also a generalized function with the singular point $o$ (the origin). Using our convolution formula for $T_\Theta(E)$, we establish a decay estimate for $T_\Theta(E)$ at infinity and present some interesting results. As applications of this asymptotic theory, we prove a weighted Brunn-Minkowski type inequality and study the solutions to the weighted Minkowski problem for pseudo-cones. Moreover, we pose an open problem regarding $T_\Theta(E)$, which we call the asymptotic Brunn-Minkowski inequality for $C$-pseudo-cones.



\end{abstract}

\section{Introduction}

The polar operator on the space $\mathcal{K}^n$ of convex bodies and its functional counterpart, the Legendre transform on the class $\text{Cvx}(\mathbb{R}^n)$ of lower semi-continuous convex functions, play an important role in convex geometry, see \cite{Schneider-book} for more details. As an abstraction, an order reversing involution on certain partially ordered sets is called a duality. All dualities on $\mathcal{K}^n$ and $\text{Cvx}(\mathbb{R}^n)$ are completely characterized by the polar operator \cite{Boroczky-Schneider-A_characterization} and the Legendre transform \cite{Artstein-Avidan-Milman-The_concept_of_duality}, respectively. In \cite{Artstein-Avidan-Milman-A_new_duality_transform,Artstein-Avidan-Milman-Hidden_structures}, Artstein-Avidan and Milman introduced a new duality transform that is different from the Legendre transform on the sub-class $\text{Cvx}_0(\mathbb{R}^n)$ (the class of geometric convex functions) of $\text{Cvx}(\mathbb{R}^n)$. Naturally, a question arises: what is the geometric version corresponding to this new duality transform? In a special case, Rashkovskii \cite{Rashkovskii-Copolar_convexity} studied the copolarity of coconvex sets in the positive orthant of $\mathbb{R}^n$. Later, Artstein-Avidan, Sadovsky and Wyczesany \cite{Artstein-Avidan-A_zoo_of_dualities} pointed out that all order reversing quasi-involutions can be induced by a cost function and exhibited various attractive dualities. Moreover, they systematically studied this geometric version of the new duality transform (termed dual polarity) and showed that the acting object of the dual polarity is a cone-like unbounded closed convex set. They also established the Blaschke-Santal\'{o} type inequality for the cone-like sets under the condition of essential symmetry. The cone-like set and dual polarity were further examined in \cite{Xu-Li-Leng-Dualities}, where the authors referred to this cone-like set as the pseudo-cone.

On the other hand, inspired by the works of Khovanskii and Timorin \cite{Khovanskii-Timorin-On_the_theory_of_coconvex_bodies} and Milman and Rotem \cite{Milman-Rotem-Complemented_Brunn_Minkowski_inequalities}, Schneider \cite{Schneider-A_Brunn_Minkowski_theory} established the Brunn-Minkowski theory for $C$-close sets, where $C$ is a pointed closed convex cone in $\mathbb{R}^n$. Similar to the dual Minkowski problem posed by Huang, Lutwak, Yang and Zhang \cite{Huang-Lutwak-Yang-Zhang-Geometric_measures} for convex bodies, Li, Ye and the fourth author \cite{Li-Ye-Zhu-The_dual_Minkowski_problem} introduced the concept of $C$-compatible sets and studied their copolarity to address the dual Minkowski problem for unbounded closed convex sets. Adopting terms and conventions from \cite{Schneider-Pseudo_cones,Schneider-A_weighted_Minkowski_theorem}, these related results have shown that dual polarity is equivalent to copolarity and that $C$-compatible sets are the same as $C$-pseudo-cones. Quite recently, Semenov and Zhao \cite{Semenov-Zhao} studied some characterizations of the $C$-asymptotic sets to solve the Minkowski problem for the $C$-asymptotic sets given by the infinity measure. Inspired by these, our first result in this paper is a classification of $C$-pseudo-cones. That is, a non-degenerated $C$-pseudo-cone can be uniquely decomposed into the sum of a $C$-asymptotic set and a point in $C$. To state this result precisely, we gather some foundational concepts, which can be found in \cite{Artstein-Avidan-A_zoo_of_dualities,Li-Ye-Zhu-The_dual_Minkowski_problem,Schneider-A_Brunn_Minkowski_theory,Schneider-Pseudo_cones,
Schneider-A_weighted_Minkowski_theorem,Schneider-Weighted_cone_volume_measures_of_pseudo_cones,Xu-Li-Leng-Dualities}.

Denote by $\langle\cdot,\cdot\rangle$ and $|\cdot|$ the inner product and norm in $\mathbb{R}^n$, respectively. We write $B_2^n$ and $\mathbb{S}^{n-1}$ for the unit ball and the unit sphere in $\mathbb{R}^n$, respectively. The topological boundary and the topological interior of a subset $E\subset\mathbb{R}^n$ are denoted by $\partial E$ and $\text{int}\,E$, respectively. Denote by $H_u^-$ a negative half-space with the outer normal vector $u\in\mathbb{S}^{n-1}$. Let $C$ be a pointed closed convex cone in $\mathbb{R}^n$ with non-empty interior, i.e., $C$ is a $n$-dimensional closed convex cone that is line-free. The polar cone of $C$ is defined by $C^\circ=\{y\in\mathbb{R}^n\,|\,\langle x,y\rangle\leqslant0,\,\forall\,x\in C\}$, which is also line-free. Denote by $\Omega_C=\mathbb{S}^{n-1}\cap\text{int}\,C$ and $\Omega_{C^{\circ}}=\mathbb{S}^{n-1}\cap\text{int}\,C^{\circ}$.
Let $o\notin\mathbb{A}\subset C$ be an unbounded closed convex set and $A=C\setminus\mathbb{A}$. If $A$ has finite volume, then $\mathbb{A}$ is called a \emph{$C$-close set} and $A$ is called a \emph{$C$-coconvex set}. Specifically, a $C$-close set $\mathbb{A}$ is called a \emph{$C$-full set} if $A$ is bounded. Moreover, $\mathbb{A}$ is called a \emph{$C$-asymptotic set} if $\mathbb{A}$ is asymptotic to $\partial C$ at infinity, i.e.,
\begin{equation}\label{limin-formula-C-asymptotic-set}
\lim_{x\in\partial\mathbb{A},|x|\rightarrow+\infty}d(x,\partial C)=0,
\end{equation}
where $d(x,\partial C)$ is the distance from $x$ to $\partial C$. Let $E\subset\mathbb{R}^n$ be a non-empty closed convex set. The recession cone of $E$ is defined by ${\rm rec}\,E=\{x\in\mathbb{R}^n\,|\,E+x\subset E\}$. If $E$ is unbounded, then ${\rm rec}\,E$ is a closed convex cone. Let $o\notin E\subset\mathbb{R}^n$ be a non-empty closed convex set, $E$ is called a \emph{pseudo-cone} if there holds $\lambda x\in E$ for any $\lambda\geqslant1$ and $x\in\,E$. For an $n$-dimensional pointed closed convex cone $C$, if $E$ is a pseudo-cone and ${\rm rec}\,E=C$, then $E$ is called a \emph{$C$-pseudo-cone}.

\vskip.1cm

Now, we present a asymptotic characterization for the $C$-pseudo-cones. For the definition of the support function of $C$-pseudo-cones, please see \eqref{Definition-the-support-function}.
\begin{theorem}\label{Theorem-the-characteristic-of-C-pseudo-cone}
Let $C$ be a pointed closed convex cone in $\mathbb{R}^n$ with non-empty interior and $o\notin E\subset\mathbb{R}^n$ be a $C$-pseudo-cone, then the following two statements are equivalent:\\
$(i)$ There is a point $z\in C$ such that $h_E(v)=\langle z,v\rangle$ for all $v\in\partial\Omega_{C^\circ}$;\\
$(ii)$ $E$ can be uniquely decomposed into $E=z+\mathbb{A}$, where $\mathbb{A}$ is a $C$-asymptotic set and $z\in C$.
\end{theorem}
We call a $C$-pseudo-cone satisfying $(i)$ or $(ii)$ in Theorem \ref{Theorem-the-characteristic-of-C-pseudo-cone} as a non-degenerated $C$-pseudo-cone. Moreover, for a non-degenerated $C$-pseudo-cone $E=z+\mathbb{A}$, we call $z$ as the $C$-starting point of $E$ and $z+C$ as the asymptotic cone of $E$. If $z=o$, then Theorem \ref{Theorem-the-characteristic-of-C-pseudo-cone} reduces to Theorem 2.6 in \cite{Semenov-Zhao}. But the methods in this paper are different from \cite{Semenov-Zhao}. There is a beautiful example of the degenerated $C$-pseudo-cone proposed by Schneider in high dimension ($n\geqslant3$), please see Figure \ref{Fig:degenerated-pseudo-cone} in Section \ref{section-asymptotic}. Furthermore, from a new perspective, the $C$-starting point of a non-degenerated $C$-pseudo-cone can offer insights into the conic linear bundle. As applications of this asymptotic theory of $C$-pseudo-cones, one can establish a strengthened version of the Brunn-Minkowski inequality for $C$-coconvex sets \cite{Schneider-A_Brunn_Minkowski_theory}, see Section \ref{Section-ABM-inequality-and-conjecture} for details.

Since the main behavior of a $C$-pseudo-cone concentrates around the origin, i.e., all $C$-pseudo-cones have almost same asymptotic behavior at infinity, the usual Lebesgue measure space is not suitable for $C$-pseudo-cones. Following Schneider's idea on the finiteness of the measure \cite{Schneider-A_weighted_Minkowski_theorem}, one can endow a weight on $C$ to ensure the compactness for $C$-pseudo-cones. Throughout this paper, we denote the weight function by $\Theta: C\setminus\{o\}\rightarrow (0,\infty)$, which is a $(-q)$-homogeneous continuous function with $q\in \mathbb{R}$. A good example of such a weight function is $\Theta(y)=|y|^{-q}$ for any $y\in C\setminus\{o\}$.

The weighted surface area measure \cite{Schneider-A_weighted_Minkowski_theorem} of $E$ is given by the push-forward measure of the weighted boundary Hausdorff measure, specifically,
\begin{equation*}
S^{\Theta}_{n-1}(E,\cdot)={\boldsymbol{\nu}_E}_{\sharp}\,\Theta(\mathcal{H}^{n-1}\llcorner\partial_eE),
\end{equation*}
where $\boldsymbol{\nu}_E$ is the Gauss image of $E$ and $\partial_eE=E\cap\text{cl}\,(C\setminus E)$ is the effective boundary of $E$ (see \cite{Li-Ye-Zhu-The_dual_Minkowski_problem} or section 2 for more details).
Schneider \cite{Schneider-A_weighted_Minkowski_theorem} showed that for $q>n-1$ and every $C$-pseudo-cone $E$, $S^{\Theta}_{n-1}(E,\cdot)$ is finite. Moreover, the weighted volume $V_{\Theta}(E)$ and the weighted co-volume $\overline{V}_{\Theta}(E)$ of $C$-pseudo-cone $E$ are defined as follows:
\begin{equation*}
V_{\Theta}(E)=\int_{E}\Theta(x)\,d\mathcal{H}^n(x),\quad \overline{V}_{\Theta}(E)=\int_{C\setminus E}\Theta(x)\,d\mathcal{H}^n(x).
\end{equation*}
Schneider \cite{Schneider-A_weighted_Minkowski_theorem} also proved that for any $n-1<q<n$ and every $C$-pseudo-cone $E$, $\overline{V}_{\Theta}(E)$ is finite.

From the asymptotic perspective of a $C$-pseudo-cone, we define the asymptotic weighted co-volume functional of non-degenerated $C$-pseudo-cones as follows: Let $z$ be the $C$-starting point of the $C$-pseudo-cone $E$, expressed as $E=z+\mathbb{A}$ for some $C$-asymptotic set $\mathbb{A}$. The asymptotic weighted co-volume of $E$ is defined by
\begin{equation*}
T_{\Theta}(E)=T_{\Theta}(\mathbb{A},z)=\int_{(z+C)\setminus E}\Theta(x)\,d\mathcal{H}^n(x).
\end{equation*}
We can view this asymptotic weighted co-volume $T_{\Theta}(E)$ as a functional of two variables. If we fix a $C$-asymptotic set $\mathbb{A}$, then $T_{\Theta}(E)=T_{\Theta}(\mathbb{A},\cdot)$ becomes a generalized function on $C$. As our second result, we summarize and enhance Schneider's finiteness theory from \cite{Schneider-A_weighted_Minkowski_theorem} into the following results.
\begin{theorem}\label{Theorem-finite-table}
Let $E$ be a $C$-pseudo-cone and $\mathbb{A}$ be a $C$-asymptotic set with $z\neq o$. Then the following table holds:
\begin{center}
\setlength{\tabcolsep}{1mm}
\renewcommand\arraystretch{2}
\begin{tabular}{|l|l|l|l|l|l|}
  \hline
  \diagbox{condition}{measure} & $S_{n-1}^{\Theta}(E,\cdot)$ & $V_{\Theta}(E)$ & $\overline{V}_{\Theta}(E)$ & $T_{\Theta}(\mathbb{A},o)$ & $T_{\Theta}(\mathbb{A},z)$\\  \hline
  $q>n$ & \textcolor{blue}{finite} & \textcolor{blue}{finite} & \textcolor{red}{$+\infty$} & \textcolor{red}{$+\infty$} & \textcolor{blue}{finite} \\  \hline
  $q=n$ & \textcolor{blue}{finite} & $+\infty$ & \textcolor{red}{$+\infty$} & $+\infty$ & \textcolor{blue}{finite} \\  \hline
  $n-1<q<n$ & \textcolor{blue}{finite} & $+\infty$ & \textcolor{blue}{finite} & finite & \textcolor{blue}{finite} \\ \hline
  $0\leqslant q\leqslant n-1$ & finite or $+\infty$ & $+\infty$ & finite or $+\infty$ & finite or $+\infty$ & finite or $+\infty$ \\
  \hline
\end{tabular}
\end{center}
\end{theorem}
In the table above, the case of $n-1<q<n$ has been solved by Schneider \cite{Schneider-A_weighted_Minkowski_theorem}. Here, we focus more on the case of $q>n$. The highlighted (in blue) entries indicate that there is useful and interesting information regarding the finiteness of the asymptotic weighted co-volume and the weighted co-volume of the $C$-pseudo-cone $E$. This suggests that as $q$ exceeds $n$, we can gain additional insights into the structure and properties of $C$-pseudo-cones. For example, we demonstrate that for every $q>n$ and any $C$-asymptotic set $\mathbb{A}$, there exists a basic yet important convolution formula (see Section 4 for details):
\begin{equation*}
T_{\Theta}(\mathbb{A},z)=\boldsymbol{\chi}_{-C}\ast\Theta\,(z)-V_{\Theta}(z+\mathbb{A}),\,z\in C,
\end{equation*}
where $\boldsymbol{\chi}_{-C}$ is the characteristic function of $-C$.
Fixing a $C$-asymptotic set $\mathbb{A}$, $T_{\Theta}(\mathbb{A},\cdot)$ is a generalized function on $C$. We establish a decay estimate for $T_{\Theta}(\mathbb{A},\cdot)$ at infinity, and then we obtain the following interesting formula:
Let $q>n$ and $o\neq z\in C$. If the weight function $\Theta$ is $C^1$-smooth, then
\begin{equation*}
\boldsymbol{\chi}_{-C}\ast\Theta\,(z)=\frac{1}{n-q}\int_{z+C}\frac{\partial\Theta}{\partial z}(x)\,d\mathcal{H}^n(x).
\end{equation*}
We believe that this imperceptible formula will be instrumental in establishing the Brunn-Minkowski inequality for $T_{\Theta}$.

Based on the fact that the Minkowski sum of two $C$-pseudo-cones is also a $C$-pseudo-cone (see Lemma 3.9 in \cite{Semenov-Zhao}), we can derive the following Brunn-Minkowski type inequality. For $0\leqslant q\leqslant n-1$, we call a $C$-pseudo-cone $E$ as a $(C,\Theta)$-close set if $\overline{V}_\Theta(E)<+\infty$.
\begin{theorem}\label{Theorem-Brunn-Minkowski-inequality-q-small-than-n-1}
Suppose that $0\leqslant q\leqslant n-1$. Let $E_1$ and $E_2$ be two $(C,\Theta)$-close sets. Then, we have
\begin{equation}\label{Formula-Brunn-Minkowski-inequality-q-small-than-n-1}
\overline{V}_\Theta(E_1+E_2)^\frac{1}{n-q}\leqslant\overline{V}_\Theta(E_1)^\frac{1}{n-q}+\overline{V}_\Theta(E_2)^\frac{1}{n-q}.
\end{equation}
Moreover, if $q\in[0,n-1)$, the equality holds in inequality \eqref{Formula-Brunn-Minkowski-inequality-q-small-than-n-1} if and only if $E_1$ and $E_2$ are dilates of each other.
\end{theorem}
From the asymptotic perspective of $C$-pseudo-cones, we conjecture that the following Brunn-Minkowski type inequality holds.
\vskip 1mm
\noindent {\bf Open problem (Asymptotic Brunn-Minkowski inequality for $C$-pseudo-cones)}:
Let $E_1,E_2$ be two non-degenerated $C$-pseudo-cones, and let $z_1,z_2\neq o$ be their respective $C$-starting points. For $\lambda\in(0,1)$ and $0\leqslant q\neq n$, there holds
\begin{equation*}
T_\Theta(\lambda E_1+(1-\lambda)E_2)^\frac{1}{n-q}\leqslant\lambda T_\Theta(E_1)^\frac{1}{n-q}+(1-\lambda)T_\Theta(E_2)^\frac{1}{n-q}?
\end{equation*}

\begin{remark}[Weighted theory v.s. dual theory] Comparing the weighted Brunn-Minkowski theory for $C$-pseudo-cones in \cite{Schneider-A_weighted_Minkowski_theorem} with the dual Brunn-Minkowski theory for $C$-compatible sets in \cite{Li-Ye-Zhu-The_dual_Minkowski_problem}, we would like to note that the dual theory is a logarithmic version of the isotropic weighted theory. The weight functions $\Theta$, which can be isotropic or anisotropic, have two representative examples in \cite{Schneider-A_weighted_Minkowski_theorem}: $|x|^{-q}$ and $\langle\mathfrak{u},x\rangle^{-q}$ for $x\in C\setminus \{o\}$ and some fixed unit vector $\mathfrak{u} \in \Omega_C\cap(-\Omega_{C^\circ})$. If we consider the weight function $\Theta=|x|^{q-n}$, then the $q$-th dual curvature measure of the $C$-pseudo-cone $E$ in \cite{Li-Ye-Zhu-The_dual_Minkowski_problem} is given by
\begin{equation*}
d\widetilde{C}_q(E,\cdot)=-\frac{1}{n}h_E(\cdot)\,dS_{n-1}^{\Theta}(E,\cdot).
\end{equation*}
The $q$-dual volume of $C$-pseudo-cone $E$ is defined as
\begin{equation*}
\widetilde{V}_q(E)=\frac{1}{n}\int_{\Omega_C^e}\rho_E^q(u)\,du.
\end{equation*}
Here, $h_E$ and $\rho_E$ denote the support function and radial function of $E$, respectively. By applying Theorem \ref{Theorem-finite-table}, we will show that $\widetilde{V}_q(E)$ is finite for every $C$-pseudo-cone $E$ and for any $q\in(-\infty,0)\cup(0,1)$ in Section \ref{section-finiteness}.
\end{remark}

Related results of the noncompact version of the classical Minkowski problem have long been studied by Pogorelov \cite{Pogorelov-An_analogue_of_the_Minkowski_problem}, Bakelman \cite{Bakelman-Variational_problems}, and Urbas \cite{Urbas-The_equation_of_prescribed_Gauss_curvature_without_boundary_conditions}. Subsequently, Chou and Wang \cite{Chou-Wang-Minkowski_problems_for_complete_noncompact} studied the smooth solutions to noncompact Minkowski problems in detail. Recently, Choi et al. \cite{Choi-Choi-Daskalopoulos-Convergence,Choi-Choi-Daskalopoulos-Uniqueness_of_ancient_solutions,Choi-Daskalopoulos-Kim-Lee-The_evolution} studied significant problems related to the evolution of hypersurfaces asymptotic to a cylinder. In cases where these hypersurfaces are asymptotic to a cone, Schneider \cite{Schneider-A_Brunn_Minkowski_theory,Schneider-A_weighted_Minkowski_theorem} was the first to investigate the noncompact Minkowski problem.
Moreover, concerning the weighted Minkowski problem for $C$-pseudo-cones, Schneider \cite{Schneider-A_weighted_Minkowski_theorem} provided an elegant result that characterizes the weighted surface area measure for $n-1<q<n$. As applications of our asymptotic theory, we present solutions to the weighted Minkowski problem for $0\leqslant q<n-1$ as follows.
\begin{theorem}\label{Weighted-Minkowski-problem-for-q-small-than-n-1}
For $q\in[0,n-1)$, given a nonzero finite Borel measure $\mu$ on $\Omega_{C^\circ}$, there exists a $(C,\Theta)$-close set $E$ such that $S^\Theta_{n-1}(E,\cdot)=\mu$. Moreover, if $K$ and $L$ are $C$-determined sets by $\omega\subset \Omega_{C^\circ}$ such that $S^\Theta_{n-1}(K,\cdot)=S^\Theta_{n-1}(L,\cdot)=\mu$, then $K=L$.
\end{theorem}

\vskip 1mm

The organization of this paper is as follows. In Section 2, we discuss the push-forward measure and co-area formula related to $C$-pseudo-cones. Section 3 presents the asymptotic theory of $C$-pseudo-cones. Section 4 summarizes results on the finiteness of the weighted co-volume and the asymptotic weighted co-volume. In Section 5, we provide integral formulas for the (asymptotic) weighted co-volume and the weighted surface area measure. Theorem \ref{Theorem-Brunn-Minkowski-inequality-q-small-than-n-1} and the asymptotic Brunn-Minkowski inequality for $C$-pseudo-cones are discussed in Section 6. Finally, we study the weighted Minkowski problem with subcritical exponent.

\vskip 5mm \noindent  {\bf Acknowledgement.}  The authors would like to thank Professor Rolf Schneider for providing a crucial example of the degenerated $C$-pseudo-cone.

\section{Background: Push-forward measure and the Co-area formula}

In this section, we introduce some tools related to $C$-pseudo-cones. Let $E\subset\mathbb{R}^n$ be a non-empty closed convex set. The recession cone of $E$ is defined as
$
{\rm rec}\,E=\{x\in\mathbb{R}^n\,|\,E+x\subset E\}.
$
If $E$ is unbounded, then ${\rm rec}\,E$ is a closed convex cone.
\begin{definition}[see \cite{Artstein-Avidan-A_zoo_of_dualities,Schneider-Pseudo_cones,Schneider-A_weighted_Minkowski_theorem,Xu-Li-Leng-Dualities}]
Let $o\notin E\subset\mathbb{R}^n$ be a non-empty closed convex set. The set $E$ is called a pseudo-cone if $\lambda x\in E$ for any $\lambda\geqslant1$ and $x\in\,E$. Crucially, we have:
\begin{equation*}
E\ \text{is a pseudo-cone if and only if}\ E\subset{\rm rec}\,E.
\end{equation*}
Moreover, if $E$ is $n$-dimensional and line-free, then ${\rm rec}\,E$ is an $n$-dimensional pointed closed convex cone and $E$ is called a $({\rm rec}\,E)$-pseudo-cone; Conversely, given an $n$-dimensional pointed closed convex cone $C$, if $E$ is a pseudo-cone and ${\rm rec}\,E=C$, then $E$ is called a $C$-pseudo-cone.
\end{definition}

Let $E$ be a $C$-pseudo-cone. The support function of $E$ is defined by
\begin{equation}\label{Definition-the-support-function}
h_E(v)=\max\{\langle x,v\rangle\,|\,x\in E\},\,v\in\Omega_{C^{\circ}}.
\end{equation}
Note that $-\infty<h_E(v)<0$ for any $v\in\Omega_{C^{\circ}}$, we denote $\overline{h}_E=-h_E$. The radial function of $E$ is defined by
\begin{equation*}
\rho_E(u)=\min\{r>0\,|\,ru\in E\},\,u\in\Omega_C.
\end{equation*}
Fixed $\mathfrak{u}\in\Omega_C\cap(-\Omega_{C^\circ})$, then $\langle\mathfrak{u},x\rangle>0$ for any $x\in C\setminus\{o\}$ (see e.g., \cite{Schneider-Pseudo_cones}). Define $C(t)=\{x\in C\,|\,\langle x,\mathfrak{u}\rangle=t\}$ and $C^-(t)=\{x\in C\,|\,\langle x,\mathfrak{u}\rangle\leqslant t\}$.

Li, Ye and the fourth author \cite{Li-Ye-Zhu-The_dual_Minkowski_problem} introduced the following notion of $C$-compatible set.
\begin{definition}[see \cite{Li-Ye-Zhu-The_dual_Minkowski_problem}]\label{Def-C-closed-convex-hull}
Let $o\notin E\subset C$ be a non-empty set. The closed convex hull of $E$ with respect to $C$ is defined by
\begin{equation*}
\begin{aligned}
{\rm conv}(E,C)&=\bigcap\{\widetilde{E}|\widetilde{E}\ \text{is a C-full set such that}\ E\subset\widetilde{E}\}\\
&=\bigcap\{\widetilde{E}|\widetilde{E}\ \text{is a C-close set such that}\ E\subset\widetilde{E}\}\\
&=\bigcap\{\widetilde{E}|\widetilde{E}\ \text{contains x+C for all}\ x\in E\}\\
&=C\cap\bigcap_{v\in\Omega_{C^\circ}}\{H_v^-|\,E\subset H_v^-\}.
\end{aligned}
\end{equation*}
If ${\rm conv}(E,C)=E$, then $E$ is called a $C$-compatible set.
\end{definition}

\begin{lemma}[see \cite{Li-Ye-Zhu-The_dual_Minkowski_problem}]
Let $o\notin E\subset C$, then $E$ is a $C$-pseudo-cone if and only if it is a $C$-compatible set.
\end{lemma}

Let $E$ be a $C$-pseudo-cone. Some geometric mappings and sets associated with $E$, such as the Gauss image map and the effective boundary of $E$, as discussed in \cite{Li-Ye-Zhu-The_dual_Minkowski_problem}, are summarized as follows:
\begin{equation*}
\qquad\quad\text{transport mappings}\left\{
\begin{aligned}
\boldsymbol{\nu}_E:\mathcal{B}(\partial_eE)&\rightarrow\mathcal{B}(\Omega_{C^\circ}^e),\,\beta\mapsto\boldsymbol{\nu}_E(\beta),\\
r_E:\Omega_C^e&\rightarrow\partial_eE,\,u\mapsto\rho_E(u)u,\\
\boldsymbol{\alpha}_E:\mathcal{B}(\Omega_C^e)&\rightarrow\mathcal{B}(\Omega_{C^\circ}^e),\,\omega\mapsto\boldsymbol{\nu}_E(r_E(\omega)),\\
\end{aligned}
\right.
\end{equation*}
\begin{equation*}
\text{transport inverse mappings}\left\{
\begin{aligned}
\boldsymbol{\nu}^*_E:\mathcal{B}(\Omega_{C^\circ}^e)&\rightarrow\mathcal{B}(\partial_eE),\,\eta\mapsto\boldsymbol{\nu}^*_E(\eta),\\
r^{-1}_E:\partial_eE&\rightarrow\Omega_C^e,\,x\mapsto\frac{x}{|x|},\\
\boldsymbol{\alpha}^*_E:\mathcal{B}(\Omega_{C^\circ}^e)&\rightarrow\mathcal{B}(\Omega_C^e),\,\eta\mapsto r^{-1}_E(\boldsymbol{\nu}^*_E(\eta)).\\
\end{aligned}
\right.
\end{equation*}
For more details, one can refer to \cite[p.\,2018-2019]{Li-Ye-Zhu-The_dual_Minkowski_problem}. Let $\Theta$ be a positive continuous function on $C\setminus\{0\}$. We consider two measures with density $\Theta$: the $\Theta$-weighted spherical Lebesgue measure, denoted by $\Theta\mu_{\mathbb{S}^{n-1}}\llcorner\Omega_C^e$ and the $\Theta$-weighted boundary Hausdorff measure, denoted by $\Theta\mathcal{H}^{n-1}\llcorner\partial_eE$. The measure $\Theta\mathcal{H}^{n-1}\llcorner\partial_eE$ can be pushed forward to the $\Theta$-weighted surface area measure $S^{\Theta}_{n-1}(E,\cdot)$ on $\Omega_{C^\circ}^e$ via the Gauss image map $\boldsymbol{\nu}_E$. Specifically, for any Borel set $\eta\in\mathcal{B}(\Omega_{C^\circ}^e)$,
\begin{equation}\label{weighted-surface-area-measure}
S^{\Theta}_{n-1}(E,\eta)={\boldsymbol{\nu}_E}_{\sharp}(\Theta\mathcal{H}^{n-1}\llcorner\partial_eE)\,(\eta)
=\int_{\boldsymbol{\nu}^*_E(\eta)\cap\partial_eE}\Theta(x)\,d\mathcal{H}^{n-1}(x).
\end{equation}
It is not hard to verify that $S^{\Theta}_{n-1}(E,\cdot)$ is a Radon measure on $\Omega_{C^\circ}^e$. By the standard approximation theory in measure theory, \eqref{weighted-surface-area-measure} is equivalent to the following: for any bounded measurable function $f$ on $\Omega_{C^\circ}^e$,
\begin{equation}\label{weighted-surface-area-measure-equivalent-formula}
\int_{\Omega_{C^\circ}^e}f(v)\,dS^{\Theta}_{n-1}(E,v)=\int_{\partial_eE}f(\nu_E(x))\Theta(x)
\,d\mathcal{H}^{n-1}(x),
\end{equation}
where $\nu_E$ is the Gauss map of $E$ defined $\mathcal{H}^{n-1}$-a.e. on $\partial_eE$.

To transform the boundary integrals and the spherical integrals, one can apply the co-area formula of Federer \cite[Theorem 3.2.22]{Federer-book} and the approximation Jacobian of the radial map. For further details, please refer to \cite[p.\,170]{Hug-Weil-book}, \cite[p.\,2022]{Li-Ye-Zhu-The_dual_Minkowski_problem}, and \cite[p.\,9]{Schneider-A_weighted_Minkowski_theorem}. To ensure rigor, we provide a simple yet important lemma:
\begin{lemma}\label{locally-Lipschitz-radial-map-of-C-pseudo-cone}
The radial map $r_E$ of the $C$-pseudo-cone $E$ is an internally closed Lipschitz homeomorphism between $\Omega_C^e$ and $\partial_eE$. Therefore, we can conclude that the convex hypersurface $\partial E$ is a spherical locally Lipschitz graph on $\Omega_C$.
\end{lemma}
\begin{proof}
There exists a hyperplane $H\cong\mathbb{R}^{n-1}$ such that $E$ can be represented by the epigraph of some convex function $f$ on $H$, with $C\setminus\{0\}\subset H^+\setminus H$. For any compact subset $\omega\subset\Omega_C$ and any $u_1,u_2\in\omega$, there exists an $(n-1)$-dimensional convex body $K\subset H$ such that
\begin{equation*}
r_E(u_i)=(x_i,f(x_i)),\,x_i\in K,\,i=1,2.
\end{equation*}
Since $f$ is a convex function on $\mathbb{R}^{n-1}$, it follows that $f$ is Lipschitz on $K$, and we denote the Lipschitz constant by $\text{Lip}(f)=C_1$. Thus, one has
\begin{equation*}
|r_E(u_1)-r_E(u_2)|=\sqrt{|x_1-x_2|^2+|f(x_1)-f(x_2)|^2}\leqslant\sqrt{1+C_1^2}\ |x_1-x_2|.
\end{equation*}
Since $\omega\subset C\setminus\{0\}\subset H^+\setminus H$ is compact and $o\notin E$, there exists a constant $C_2$, depending only on $\omega$ and $E$, such that $\frac{1}{C_2}<\rho_E(u)<C_2$ for all $u\in\omega$. Note that $u_i$ can be represented as $u_i=\bigg(\frac{1}{\rho_E(u_i)}x_i,\sqrt{1-\big(\frac{|x_i|}{\rho_E(u_i)}\big)^2}\bigg)$ for $i=1,2$. Thus
\begin{equation*}
|u_1-u_2|\geqslant\Big|\frac{1}{\rho_E(u_1)}x_1-\frac{1}{\rho_E(u_2)}x_2\Big|
\geqslant\Big|\frac{1}{C_2}x_1-\frac{1}{C_2}x_2\Big|.
\end{equation*}
Therefore, one has
\begin{equation*}
|r_E(u_1)-r_E(u_2)|\leqslant\sqrt{1+C_1^2}\ |x_1-x_2|\leqslant C_2\sqrt{1+C_1^2}\ |u_1-u_2|,
\end{equation*}
which shows that $r_E$ is a Lipschitz mapping on $\omega$, i.e., $r_E$ is internally closed Lipschitz on $\Omega_C^e$. It is clear that $r^{-1}_E$ is Lipschitz on $\partial_eE$.
\end{proof}
Due to Lemma \ref{locally-Lipschitz-radial-map-of-C-pseudo-cone}, $r_E$ is differentiable almost everywhere on $\Omega_C^e$ and $r^{-1}_E$ is differentiable almost everywhere on $\partial_eE$. Let $v\in\Omega_{C^\circ}$ be a regular normal vector of $E$, $x=\boldsymbol{\nu}_E^*(v)$ and $u=r_E^{-1}(x)$. By establishing the orthonormal frame $\{e_1,\cdots,e_{n-1}\}$ of $\mathbb{S}^{n-1}$ at $v$, we obtain the orthogonal decomposition:
\begin{equation*}
\left\{
\begin{aligned}
u&=\sum_{i=1}^{n-1}\langle u,e_i\rangle e_i+\langle u,v\rangle v\triangleq\sum_{i=1}^{n-1}s_ie_i+sv,\\
x&=\rho_E(u)u=\sum_{i=1}^{n-1}\langle x,e_i\rangle e_i+\langle x,v\rangle v\triangleq\sum_{i=1}^{n-1}h_ie_i+h_E(v)v,
\end{aligned}
\right.
\end{equation*}
where $h_i$ is the covariant partial derivative of the support function $h$ with respect to the frame $\{e_1,\cdots,e_{n-1}\}$. Denote by $\text{Id}$ the identity mapping on $T_x\partial E$, then the tangent mapping of $r_E^{-1}$ at $x$ is given by (noting that $r^{-1}_E$ is differentiable $\mathcal{H}^{n-1}$-a.e. on $\partial_eE$):
\begin{equation*}
d\Big(\frac{x}{|x|}\Big)=\frac{1}{|x|}(\text{Id}-u\otimes(u\otimes\text{Id})).
\end{equation*}
Therefore, using the Gram matrix of $\Big\{d\Big(\frac{x}{|x|}\Big)(e_1),\cdots,d\Big(\frac{x}{|x|}\Big)(e_{n-1})\Big\}$, we have
\begin{equation*}
\Big|d\Big(\frac{x}{|x|}\Big)(e_1)\wedge\cdots\wedge d\Big(\frac{x}{|x|}\Big)(e_{n-1})\Big|=\frac{\overline{h}}{|x|^n},
\end{equation*}
thus the $(\mathcal{H}^{n-1},n-1)$-approximate Jacobian of $r_E^{-1}$ at $\mathcal{H}^{n-1}$-a.e. $x\in\partial_eE$ is
\begin{equation*}
ap\,J_{n-1}r_E^{-1}(x)=\Big|d\Big(\frac{x}{|x|}\Big)(e_1)\wedge\cdots\wedge d\Big(\frac{x}{|x|}\Big)(e_{n-1})\Big|
=\frac{|h_E(v)|}{|x|^n}=\frac{|\langle v,u\rangle|}{\rho_E^{n-1}(u)}.
\end{equation*}
Since $r_E\circ r_E^{-1}=\text{Id}$, the $(\mathcal{H}^{n-1},n-1)$-approximate Jacobian of $r_E$ at $\mu_{\mathbb{S}^{n-1}}$-a.e. $u\in\Omega_C^e$ is
\begin{equation*}
ap\,J_{n-1}r_E(u)=(ap\,J_{n-1}r_E^{-1}(x))^{-1}=\frac{\rho_E^{n-1}(u)}{|\langle v,u\rangle|}=\frac{\rho_E^n(u)}{|\langle x,v\rangle|}.
\end{equation*}
Applying the co-area formula in \cite[Theorem 3.2.22]{Federer-book}, for any integrable functions $f$ on $\Omega_C^e$ and $g$ on $\partial_eE$, we have
\begin{equation}\label{The-co-area-formula-in-convex-geometry}
\left\{
\begin{aligned}
\int_{\Omega_C^e}f(u)\frac{\rho_E^n(u)}{\overline{h}_E(\alpha_E(u))}\,d\mu_{\mathbb{S}^{n-1}}(u)=\int_{\partial_eE}f\Big(\frac{x}{|x|}\Big)
\,d\mathcal{H}^{n-1}(x),\\
\int_{\partial_eE}g(x)\frac{\overline{h}_E(\nu_E(x))}{|x|^n}\,d\mathcal{H}^{n-1}(x)=\int_{\Omega_C^e}g(r_E(u))\,d\mu_{\mathbb{S}^{n-1}}(u).\\
\end{aligned}
\right.
\end{equation}

\section{Asymptotic analysis of $C$-pseudo-cones} \label{section-asymptotic}

As an unbounded closed convex set in $C$, the $C$-pseudo-cone is a central research object in unbounded Brunn-Minkowski theory. In this section, we will discuss some asymptotic properties of $C$-pseudo-cones. These properties are useful for establishing the Brunn-Minkowski theory for more general unbounded closed convex sets.

\begin{lemma}\label{C-asymptotic-set}
If $\mathbb{A}$ is a $C$-asymptotic set, then $\mathbb{A}$ is a $C$-pseudo-cone and
\begin{equation*}
\mathbb{A}=\bigcap_{v\in\Omega_{C^\circ}}H^-(v,h_{\mathbb{A}}(v)).
\end{equation*}
\end{lemma}
\begin{proof}
For a $C$-asymptotic set $\mathbb{A}$, its normal vectors belong to $\Omega_{C^\circ}\cup\partial\Omega_{C^\circ}$ due to formula \eqref{limin-formula-C-asymptotic-set}. Since $\mathbb{A}$ is also a closed convex set and $h_{\mathbb{A}}(w)=0$ for any $w\in\partial\Omega_{C^\circ}$, we have
\begin{equation*}
\begin{aligned}
\mathbb{A}&=\bigcap_{v\in\Omega_{C^\circ}\cup\partial\Omega_{C^\circ}}H^-(v,h_{\mathbb{A}}(v))\\
&=\bigg(\bigcap_{w\in\partial\Omega_{C^\circ}}H^-(w,h_{\mathbb{A}}(w))\bigg)\bigcap
\bigg(\bigcap_{v\in\Omega_{C^\circ}}H^-(v,h_{\mathbb{A}}(v))\bigg)\\
&\supset C\cap\bigcap_{v\in\Omega_{C^\circ}}H^-(v,h_{\mathbb{A}}(v)).
\end{aligned}
\end{equation*}
Combining this with Definition \ref{Def-C-closed-convex-hull}, we have
\begin{equation*}
\begin{aligned}
\text{conv}(\mathbb{A},C)&=C\cap\bigcap_{v\in\Omega_{C^\circ}}\{H_v^-|\,\mathbb{A}\subset H_v^-\}\\
&\subset C\cap\bigcap_{v\in\Omega_{C^\circ}}H^-(v,h_{\mathbb{A}}(v))\\
&\subset\mathbb{A}\\
&\subset\text{conv}(\mathbb{A},C),
\end{aligned}
\end{equation*}
which shows that $\mathbb{A}$ is a $C$-compatible set ($\mathbb{A}$ is also a $C$-pseudo-cone), and hence
\begin{equation*}
\mathbb{A}=C\cap\bigcap_{v\in\Omega_{C^\circ}}H^-(v,h_{\mathbb{A}}(v)).
\end{equation*}

Suppose there exists a point $x\in\bigcap_{v\in\Omega_{C^\circ}}H^-(v,h_{\mathbb{A}}(v))$ but $x\notin C$.Let $p_C(x)$ and $p_{C^\circ}(x)$ denote the metric projection of $x$ on $C$ and $C^\circ$, respectively. Define $d(x,C)=|x-p_C(x)|$. Let
\begin{equation*}
\left\{
\begin{aligned}
\alpha&=\arctan\frac{\delta d(x,C)}{\sqrt{|x|^2-d(x,C)^2}},\\
\beta&=\arctan\frac{d(x,C)}{\sqrt{|x|^2-d(x,C)^2}},
\end{aligned}
\right.
\end{equation*}
where $\delta>0$ is sufficiently small such that
\begin{equation*}
v_0=\frac{p_{C^\circ}(x)-|p_{C^\circ}(x)|(\tan\alpha)\frac{p_C(x)}{|p_C(x)|}}{\sqrt{(|p_{C^\circ}(x)|\tan\alpha)^2+
|p_{C^\circ}(x)|^2}}\in\Omega_{C^\circ}.
\end{equation*}
Thus, $x\in H^-(v_0,h_{\mathbb{A}}(v_0))$, which implies $\langle x,v_0\rangle\leqslant h_{\mathbb{A}}(v_0)<0$. However,
\begin{equation*}
\begin{aligned}
\langle x,v_0\rangle&=\frac{\langle x,p_{C^\circ}(x)\rangle-|p_{C^\circ}(x)|(\tan\alpha)
\frac{\langle x,p_C(x)\rangle}{|p_C(x)|}}{\sqrt{(|p_{C^\circ}(x)|\tan\alpha)^2+|p_{C^\circ}(x)|^2}}\\
&=\frac{|x||p_{C^\circ}(x)|\cos(\frac{\pi}{2}-\beta)-|p_{C^\circ}(x)|(\tan\alpha)\frac{|x||p_C(x)|\cos\beta}{|p_C(x)|}}
{\sqrt{(|p_{C^\circ}(x)|\tan\alpha)^2+|p_{C^\circ}(x)|^2}}\\
&=\frac{|x||p_{C^\circ}(x)|}{\sqrt{(|p_{C^\circ}(x)|\tan\alpha)^2+|p_{C^\circ}(x)|^2}}(\sin\beta-\tan\alpha\cos\beta)>0,
\end{aligned}
\end{equation*}
where we used
\begin{equation*}
\sin\beta-\tan\alpha\cos\beta=\cos\beta(\tan\beta-\tan\alpha)=\frac{(1-\delta)d(x,C)\cos\beta}{\sqrt{|x|^2-d(x,C)^2}}>0.
\end{equation*}
This contradiction shows that $\bigcap_{v\in\Omega_{C^\circ}}H^-(v,h_{\mathbb{A}}(v))\subset C$, which implies
\begin{equation*}
\mathbb{A}=C\cap\bigcap_{v\in\Omega_{C^\circ}}H^-(v,h_{\mathbb{A}}(v))=\bigcap_{v\in\Omega_{C^\circ}}H^-(v,h_{\mathbb{A}}(v)).
\end{equation*}
\end{proof}

Let $E$ be a $C$-pseudo-cone and $p_E(o)$ be the metric projection of the origin $o$ on $E$, $u\in\partial C\cap\mathbb{S}^{n-1}$. Denote by $l_{x,u}=\{x+\lambda u\,|\,\forall\ \lambda>0,\ x\in\mathbb{R}^n\ \text{and}\ u\in\mathbb{S}^{n-1}\}$ the ray starting at $x\in\mathbb{R}^n$ in the direction $u\in\mathbb{S}^{n-1}$. If $E\subset x_0+C$ for some $x_0\in C$, we define the following set
\begin{equation*}
\partial E|_u=\{x\in\partial E\,|\,\exists\lambda,\mu\geqslant0\ \text{such that}\,  x=x_0+\lambda(p_E(o)-x_0)+\mu u\},
\end{equation*}
then there exists a boundary partition
\begin{equation*}
\partial E=\bigcup_{u\in\partial C\cap\mathbb{S}^{n-1}}\partial E|_u.
\end{equation*}
Moreover, we define the function $f_{u,x_0}$ as $f_{u,x_0}(|x|)=d(x,l_{x_0,u}),\ x\in\partial E|_u$.

\begin{lemma}\label{uniformly-lemma}
Let $E$ be a $C$-pseudo-cone, $u\in\partial C\cap\mathbb{S}^{n-1}$ and $x_0\in C$. If $E\subset x_0+C$, then the function $f_{u,x_0}$ is monotonically decreasing. Therefore, there exists the following limiting function
\begin{equation}\label{limit-decreasing}
\lim_{x\in\partial E|_u,\,|x|\rightarrow+\infty}f_{u,x_0}(|x|)\triangleq \widetilde{f}_{u,x_0}.
\end{equation}
Moreover, the following holds
\begin{equation*}
f_{u,x_0}(|x|)-\widetilde{f}_{u,x_0}\rightrightarrows0,\ \mbox{uniformly on}\ \partial\Omega_C,\ \mbox{as}\ |x|\rightarrow+\infty.
\end{equation*}
\end{lemma}
\begin{proof}
For any $u\in\partial C\cap\mathbb{S}^{n-1}$, let $x_1,x_2\in\partial E|_u$ with $|x_1|<|x_2|$. Since a $C$-pseudo-cone is the same as a $C$-compatible set and given that $x_1\in E$, we have
\begin{equation*}
E=\text{conv}(E,C)\supset\text{conv}(\{x_1\},C)=x_1+C\supset\{x_1+\lambda u\,|\,\lambda>0\}=l_{x_1,u}.
\end{equation*}
Combining with $|x_1|<|x_2|$, we find that the ray $l_{x_2,u}$ lies in between $l_{x_0,u}$ and $l_{x_1,u}$. Thus,
\begin{equation*}
f_{u,x_0}(|x_1|)\geqslant f_{u,x_0}(|x_2|).
\end{equation*}

For $x\in\partial E|_u$ with $|x|=N>0$, we denote by $\alpha(u,N)$ the angle between $u$ and $x-x_0$, then $f_{u,x_0}(|x|)=|x-x_0|\,\sin\alpha(u,N)$. Let $p=\frac{p_E(o)-x_0}{|p_E(o)-x_0|}$, and let $\widehat{p\,u}$ be the geodesic curve between $p$ and $u$ on $\mathbb{S}^{n-1}$. By the formula \eqref{limit-decreasing}, we have $\widetilde{f}_{u,x_0}=d(l_{x_0,u},\partial E|_u)$. Due to the continuity of the boundary of the closed convex sets $C$ and $E$, $\widetilde{f}_{u,x_0}$ is continuous with respect to $u$. Similarly, $\sin\alpha(u,N)$ is also continuous with respect to $u$. Therefore, for any $\varepsilon>0$, there exists an open neighborhood $U(\widehat{p\,u})$ of the geodesic curve $\widehat{p\,u}$ in $\mathbb{S}^{n-1}$ such that
\begin{equation}\label{sin-function-estimates}
|\sin\alpha(u,N)-\sin\alpha(v,N)|<\varepsilon
\end{equation}
and
\begin{equation}\label{tilde-f-function-estimates}
|\widetilde{f}_{u,x_0}-\widetilde{f}_{v,x_0}|<\varepsilon
\end{equation}
for all $v\in U(\widehat{p\,u})\cap\partial\Omega_C$ and $N>|p_E(o)|$. Now, the collection $\{U(\widehat{p\,u})\,|\,u\in\partial\Omega_C\}$ forms a family of open covers of the spherical closed convex set $\mathbb{S}^{n-1}\cap C=\bigcup_{u\in\partial\Omega_C}\widehat{p\,u}$. By the Heine-Borel finite covering theorem, there exists a finite sub-family $\{U(\widehat{p\,u_i})\,|\,u_i\in\partial\Omega_C\}_{i=1}^m$ of $\{U(\widehat{p\,u})\,|\,u\in\partial\Omega_C\}$ such that
\begin{equation*}
\mathbb{S}^{n-1}\cap C\subset\bigcup_{i=1}^mU(\widehat{p\,u_i}).
\end{equation*}
Hence, for any $v\in\partial\Omega_C$, there exists an $u_{i_v}$ such that $v\in U(\widehat{p\,u_{i_v}})$ for some $1\leqslant i_v\leqslant m$.

For the above $\varepsilon>0$ and for every $u_i\in\partial\Omega_C$, by the formula \eqref{limit-decreasing}, we have
\begin{equation*}
\lim_{x\in\partial E|_{u_i},\,|x|\rightarrow+\infty}\big(f_{u_i,x_0}(|x|)-\widetilde{f}_{u_i,x_0}\big)=0.
\end{equation*}
Thus, there exists $N_i>0$ such that for all $x\in\partial E|_{u_i}$ with $|x|\geqslant N_i$, the following holds
\begin{equation}\label{u-i-function-estimates}
f_{u_i,x_0}(|x|)-\widetilde{f}_{u_i,x_0}<\varepsilon.
\end{equation}
Now, let $N_0=\max\{N_1,\cdots,N_m\}$. Consider $v\in\partial\Omega_C$ and $x\in\partial E|_v$ with $|x|\geqslant N_0$.\\
Case $1$: If $v=u_i$ for some $1\leqslant i\leqslant m$, then by \eqref{u-i-function-estimates}, we have
\begin{equation*}
f_{v,x_0}(|x|)-\widetilde{f}_{v,x_0}<\varepsilon.
\end{equation*}
Case $2$: If $v\in U(\widehat{p\,u_{i_v}})$ for some $1\leqslant i_v\leqslant m$, then
\begin{equation*}
\begin{aligned}
f_{v,x_0}(|x|)-\widetilde{f}_{v,x_0}&\leqslant f_{v,x_0}(N_0)-\widetilde{f}_{v,x_0}
=N_0\,\sin\alpha(v,N_0)-\widetilde{f}_{v,x_0}\\
&=N_0\big(\sin\alpha(v,N_0)-\sin\alpha(u_{i_v},N_0)+\sin\alpha(u_{i_v},N_0)\big)\\
&\quad-\widetilde{f}_{u_{i_v},x_0}+\widetilde{f}_{u_{i_v},x_0}-\widetilde{f}_{v,x_0}\\
&=N_0\big(\sin\alpha(v,N_0)-\sin\alpha(u_{i_v},N_0)\big)+(f_{u_{i_v},x_0}(N_0)-\widetilde{f}_{u_{i_v},x_0})\\
&\quad+(\widetilde{f}_{u_{i_v},x_0}-\widetilde{f}_{v,x_0})\\
&<(N_0+2)\varepsilon,
\end{aligned}
\end{equation*}
where we used \eqref{sin-function-estimates}, \eqref{tilde-f-function-estimates} and \eqref{u-i-function-estimates}.

Therefore, for the above $\varepsilon>0$, if $|x|\geqslant N_0$, then
\begin{equation*}
0\leqslant f_{v,x_0}(|x|)-\widetilde{f}_{v,x_0}<(N_0+2)\varepsilon
\end{equation*}
for all $v\in\partial\Omega_C$. This shows that $f_{v,x_0}(|x|)-\widetilde{f}_{v,x_0}\rightrightarrows0\ \text{uniformly on}\ \partial\Omega_C$ as $x\in\partial E$ and $|x|\rightarrow+\infty$.
\end{proof}

\begin{definition}[$C$-starting point]\label{Definition-C-starting-point}
Let $E$ be a $C$-pseudo-cone and $z\in C$. We call $z$ as the $C$-starting point of $E$ if $E\subset z+C$ and
\begin{equation*}
f_{u,z}(|x|)\rightrightarrows0,\,\text{uniformly on}\ \partial\Omega_C,
\end{equation*}
as $|x|\rightarrow+\infty$ with $\,x\in\partial E|_u$.
\end{definition}
\begin{remark}
If $z$ is a $C$-starting point of $C$-pseudo-cone $E$, then $z$ is unique. To do so, let $p_1$ and $p_2$ be two $C$-starting points of $E$. For some $\bar{u}\in\partial\Omega_C$, we have
\begin{equation*}
f_{\bar{u},p_1}(|x|)=d(x,l_{p_1,\bar{u}})\rightarrow0,\,f_{\bar{u},p_2}(|x|)=d(x,l_{p_2,\bar{u}})\rightarrow0,
\end{equation*}
$\text{as}\ |x|\rightarrow+\infty \, \text{with}\ x\in\partial E|_{\bar{u}}$. By the uniqueness of limits, it follows that $l_{p_1,\bar{u}}=l_{p_2,\bar{u}}$. Therefore, we conclude that $p_1=p_2$.
\end{remark}

\begin{lemma}\label{Lemma-equivalence-support-function-and-C-starting-point}
Let $E$ be a $C$-pseudo-cone and $z\in C$, then
\begin{equation*}
h_E(v)=\langle z,v\rangle\ \mbox{for all}\ v\in\partial\Omega_{C^\circ}\ \mbox{if and only if}\ z\ \mbox{is the C-starting point of}\ E.
\end{equation*}
\end{lemma}
\begin{proof}
$(i)$ Without loss of generality, let $E$ be not a closed convex cone. Note that
\begin{equation*}
C=\bigcap_{v\in\partial\Omega_{C^\circ}}H^-(v,0),
\end{equation*}
we have
\begin{equation*}
\begin{aligned}
z+C&=z+\bigcap_{v\in\partial\Omega_{C^\circ}}H^-(v,0)=\bigcap_{v\in\partial\Omega_{C^\circ}}(z+H^-(v,0))\\
&=\bigcap_{v\in\partial\Omega_{C^\circ}}H^-(v,\langle z,v\rangle)=\bigcap_{v\in\partial\Omega_{C^\circ}}H^-(v,h_E(v))\triangleq C_E,
\end{aligned}
\end{equation*}
where $z+H^-(v,0)=H^-(v,\langle z,v\rangle)$ can be checked easily. Since $E\subset H^-(v,h_E(v))$ for all $v\in\partial\Omega_{C^\circ}$, one has
\begin{equation*}
E\subset C_E=z+C.
\end{equation*}
By Lemma \ref{uniformly-lemma}, we have
\begin{equation*}
f_{u,z}(|x|)-\widetilde{f}_{u,z}\rightrightarrows0,\,\text{uniformly on}\ \partial\Omega_C\ \text{as}\ |x|\rightarrow+\infty\ \text{with}  \ x\in\partial E|_u.
\end{equation*}
We claim that $\widetilde{f}_{u,z}=0$ for any $u\in\partial\Omega_C$. To do so, assume there exists $u_0\in\partial\Omega_C$ such that $\widetilde{f}_{u_0,z}>0$. Denote by $\gamma$ the angle between $u_0$ and $p_E(o)-z$, then the ray
\begin{equation*}
\widetilde{l}_{u_0}=\frac{\widetilde{f}_{u_0,z}}{|p_E(o)-z|\sin\gamma}\big(p_E(o)-z-(|p_E(o)-z|\cos\gamma)u_0\big)
+l_{z,u_0}\triangleq a+l_{z,u_0}
\end{equation*}
is an asymptotic line of $\partial E|_{u_0}$ in the $2$-dimensional plane $L=z+\text{span}\{u_0,p_E(o)-z\}$. Thus,
\begin{equation*}
\Big(\frac{1}{2}a+l_{z,u_0}\Big)\cap E=\emptyset.
\end{equation*}
By Theorem 1.3.7 in \cite{Schneider-book}, the ray $(\frac{1}{2}a+l_{z,u_0})$ and the set $E$ can be separated by a hyperplane $H(w)$ with the normal $w$, where $E\subset H^-(w)$. Clearly, $(\frac{1}{2}a+l_{z,u_0})$ is parallel to $H(w)$. If $w\in\Omega_{C^\circ}$, then $\langle u_0,w\rangle<0$. Since $\frac{1}{2}a+l_{z,u_0}\subset H^+(w)$, we have
\begin{equation*}
\frac{1}{2}\langle a,w\rangle+\lambda\langle u_0,w\rangle=\Big\langle\frac{1}{2}a+\lambda u_0,w\Big\rangle\geqslant0.
\end{equation*}
Letting $\lambda\rightarrow+\infty$, leads to the conclusion $-\infty\geqslant0$, a contradiction. Thus, we conclude that $w\in\partial\Omega_{C^\circ}$.

Next, we claim that $l_{z,u_0}\subset H(w)$. Suppose $l_{z,u_0}\nsubseteq H^-(w)$, then $l_{z,u_0}\subset H^+(w)\setminus H(w)$. Since $w\in\partial\Omega_{C^\circ}$ and $E\in H^-(w)$, it follows that $H^-(w,h_E(w))\subset H^-(w)$. This gives
\begin{equation*}
l_{z,u_0}\subset z+C=C_E=\bigcap_{v\in\partial\Omega_{C^\circ}}H^-(v,h_E(v))\subset H^-(w,h_E(w))\subset H^-(w),
\end{equation*}
which is a contradiction. If $l_{z,u_0}\subset H^-(w)\setminus H(w)$, then $\langle l_{z,u_0},w\rangle<\tau$ for some $\tau\in\mathbb{R}$. In the plane $L$, since $\frac{1}{2}a+l_{z,u_0}\subset H^+(w)\cap L$, the point $p_E(o)\notin\frac{1}{2}a+l_{z,u_0}$ lies on the other side satisfies $\langle p_E(o),w\rangle>\tau$, which contradicts the fact that $p_E(o)\in E\subset H^-(w)$. Thus, we conclude that $l_{z,u_0}\subset H(w)$.

Let $v_0\in\partial\Omega_{C^\circ}$ satisfy $\langle u_0,v_0\rangle=0$. We can choose a geodesic curve $\widehat{wv_0}$ connecting $w$ and $v_0$ on $\partial\Omega_{C^\circ}$. For any $v\in\widehat{wv_0}$, we have
\begin{equation*}
l_{z,u_0}\subset z+C=C_E\subset H^-(v,h_E(v)).
\end{equation*}
Noting that $l_{z,u_0}\subset H(v_0,h_E(v_0))$, we conclude that $H(v,h_E(v))\cap C_E=l_{z,u_0}$ for $v\in\widehat{wv_0}\setminus\{w,v_0\}$. By the properties of $H(v,h_E(v))$, there exists a sequence $\{x_i\}_{i=1}^{+\infty}\subset E$ such that
\begin{equation}\label{Contradiction-d-x-i-l-0}
d(x_i,l_{z,u_0})\rightarrow0 \ \text{as}\ i\rightarrow+\infty.
\end{equation}
Since $\widetilde{f}_{u,x_0}=d(l_{x_0,u},\partial E|_u)$ is continuous with respect to $u$, there is a small neighbourhood $U(u_0)$ of $u_0$ such that $\widetilde{f}_{u,x_0}>\delta$ for some $\delta>0$ and any $u\in U(u_0)$. Consequently, by the sufficiently small neighborhood $U(u_0)$ and $C_E\subset H^-(u_0,h_E(u_0))$, we have
\begin{equation}\label{Contradiction-partial-E-H-big-delta}
d\Big(\bigcup_{u\in U(u_0)}\partial E|_u,H(u_0,h_E(u_0))\Big)>\frac{\delta}{2}.
\end{equation}
Using the convexity of $E$ and the convergence \eqref{Contradiction-d-x-i-l-0}, we have
\begin{equation*}
\text{conv}\Big\{x_i,\bigcup_{u\in U(u_0)}\partial E|_u\Big\}\triangleq K_i\subset E\,\ \mbox{and}\,\ d(K_i,l_{z,u_0})\rightarrow0,
\ \mbox{as}\ i\rightarrow+\infty.
\end{equation*}
Note that $l_{z,u_0}\subset H(u_0,h_E(u_0))$, the above results contradict with \eqref{Contradiction-partial-E-H-big-delta}. Thus, we conclude that
\begin{equation*}
\widetilde{f}_{u,z}\equiv0.
\end{equation*}
In other words, we have $f_{u,z}(|x|)\rightrightarrows0 \,\text{uniformly on}\ \partial\Omega_C\ \text{as}\ |x|\rightarrow+\infty \, \text{with}\ x\in\partial E|_u$. Therefore, $z$ is indeed the $C$-starting point of $E$.

$(ii)$ Let $z\in C$ be the $C$-starting point of the $C$-pseudo-cone $E$, then $E\subset z+C$ and
\begin{equation}\label{starting-point-f-u-z-0}
f_{u,z}(|x|)\rightrightarrows0,\,\text{uniformly on}\ \partial\Omega_C,
\end{equation}
as $|x|\rightarrow+\infty$ with $\,x\in\partial E|_u$. Since $E\subset z+C$, one has $h_E(v)\leqslant h_{z+C}(v)=\langle z,v\rangle$ for all $v\in\partial\Omega_{C^\circ}$. Suppose that there exists a point $v_0\in\partial\Omega_{C^\circ}$ such that
\begin{equation*}
h_E(v_0)<h_{z+C}(v_0)=\langle z,v_0\rangle,
\end{equation*}
then one has
\begin{equation*}
d(l_{z,u_0},H(v_0,h_E(v_0))=\langle z,v_0\rangle-h_E(v_0)>0,
\end{equation*}
where $u_0\in\partial\Omega_C$ satisfies $u_0\perp v_0$. Since $E\subset H^-(v_0,h_E(v_0))$, for all $x\in\partial E|_{u_0}$, one has
\begin{equation*}
f_{u_0,z}(|x|)\geqslant d(l_{z,u_0},H(v_0,h_E(v_0))>0.
\end{equation*}
This contradicts with \eqref{starting-point-f-u-z-0}. Therefore, $h_E(v)=\langle z,v\rangle$ for all $v\in\partial\Omega_{C^\circ}$.
\end{proof}

\begin{lemma}\label{Lemma-equivalence-C-starting-point-and-C-asymptotic-set}
Let $o\notin E\subset C$ and $z\in C$. If $E$ is a $C$-asymptotic set, then $z+E$ is a $C$-pseudo-cone and $z$ is the $C$-starting point of $z+E$. Conversely, if $E$ is a $C$-pseudo-cone and $z$ is the $C$-starting point of $E$, then $-z+E$ is a $C$-asymptotic set.
\end{lemma}
\begin{proof}
$(i)$ Suppose that $E$ is a $C$-asymptotic set and $z\in C$. By Lemma \ref{C-asymptotic-set}, we have
\begin{equation*}
\begin{aligned}
z+E&=z+\bigcap_{u\in\Omega_{C^\circ}}H^-(u,h_E(u))\\
&=\bigcap_{u\in\Omega_{C^\circ}}\big(z+H^-(u,h_E(u))\big)=\bigcap_{u\in\Omega_{C^\circ}}H^-(u,h_{z+E}(u)).
\end{aligned}
\end{equation*}
Since $z+E\subset C$ is an unbounded closed convex set, by Definition \ref{Def-C-closed-convex-hull}, one has
\begin{equation*}
\begin{aligned}
\text{conv}(z+E,C)&=C\cap\bigcap_{u\in\Omega_{C^\circ}}\{H_u^-|\,z+E\subset H_u^-\}
=C\cap\bigcap_{u\in\Omega_{C^\circ}}H^-(u,h_{z+E}(u))\\
&=C\cap(z+E)=z+E\subset\text{conv}(z+E,C),
\end{aligned}
\end{equation*}
which shows $z+E=\text{conv}(z+E,C)$. Therefore, $z+E$ is a $C$-compatible set, i.e., $z+E$ is a $C$-pseudo-cone. Since $E$ is a $C$-asymptotic set, one has
\begin{equation}\label{asymptotic-z+x-z+partial-C}
\lim_{x\in\partial E,|x|\rightarrow+\infty}d(z+x,z+\partial C)=0.
\end{equation}
By Lemma \ref{uniformly-lemma}, there holds $f_{u,z}(|x|)-\widetilde{f}_{u,z}\rightrightarrows0\ \mbox{uniformly on}\ \partial\Omega_C,\ \mbox{as}\ |x|\rightarrow+\infty$. Suppose that $\widetilde{f}_{u',z}>0$ for some $u'\in\partial\Omega_C$, then there is a neighborhood $U(u')\subset\partial\Omega_C$ of $u'$ such that $\widetilde{f}_{u,z}>\delta$ for all $u\in U(u')$ and some $\delta>0$ from the continuity of the distance function $\widetilde{f}_{u,z}$. Therefore, one has
\begin{equation*}
d\Big(\bigcup_{u\in U(u')}\partial(z+C)|_u,\bigcup_{u\in U(u')}\partial E|_u\Big)\geqslant\delta.
\end{equation*}
This contradicts with \eqref{asymptotic-z+x-z+partial-C}. Thus, $f_{u,z}(|x|)\rightrightarrows0\ \mbox{uniformly on}\ \partial\Omega_C,\ \mbox{as}\ |x|\rightarrow+\infty$, i.e., $z$ is the $C$-starting point of $z+E$.

$(ii)$ Let $E$ be a $C$-pseudo-cone and $z$ be the $C$-starting point of $E$. By Definition \ref{Definition-C-starting-point}, we have $E\subset z+C$, i.e., $-z+E\subset C$ and
\begin{equation*}
f_{u,z}(|x|)=d(x,l_{z,u})\rightrightarrows0 \,\text{uniformly on}\ \partial\Omega_C,
\end{equation*}
$\text{as}\ |x|\rightarrow+\infty \ \text{with}\ \,x\in\partial E|_u$. Thus, for any given $\varepsilon>0$, there exists $N>0$ such that if $|x|>N$ and $x\in\partial E|_u$, then $d(x,l_{z,u})<\varepsilon$ for all $u\in\partial\Omega_C$. Hence, if $|x|>N$ and $x\in\partial E$, there exists some $\bar{u}\in\partial\Omega_C$ such that
\begin{equation*}
d(x,z+\partial C)\leqslant d(x,l_{z,\bar{u}})<\varepsilon.
\end{equation*}
This induces
\begin{equation*}
\lim_{\substack{|x|\rightarrow+\infty\\x\in\partial(-z+E)}}d(x,\partial C)=\lim_{\substack{|x|\rightarrow+\infty\\x\in\partial E}}d(-z+x,\partial C)=\lim_{\substack{|x|\rightarrow+\infty\\x\in\partial E}}d(x,z+\partial C)=0,
\end{equation*}
then $-z+E$ is a $C$-asymptotic set.
\end{proof}

Here we gave the following results as an expansion to Theorem \ref{Theorem-the-characteristic-of-C-pseudo-cone}.
\begin{theorem}\label{Theorem-the-characteristic-of-C-pseudo-cone-extension}
Let $C$ be a pointed closed convex cone in $\mathbb{R}^n$ with non-empty interior and $o\notin E\subset\mathbb{R}^n$ be a $C$-pseudo-cone, then the following three statements are equivalent:\\
$(i)$ There is a point $z\in C$ such that $h_E(v)=\langle z,v\rangle$ for all $v\in\partial\Omega_{C^\circ}$;\\
$(ii)$ There is a point $z\in C$ such that $z$ is the $C$-starting point of $E$;\\
$(iii)$ $E$ can be uniquely decomposed into $E=z+\mathbb{A}$, where $\mathbb{A}$ is a $C$-asymptotic set and $z\in C$.
\end{theorem}
\begin{proof}
Lemma \ref{Lemma-equivalence-support-function-and-C-starting-point} shows that $(i)$ is equivalent to $(ii)$. Now, we assume that $E$ satisfies $(i)$ or $(ii)$. By Lemma \ref{Lemma-equivalence-C-starting-point-and-C-asymptotic-set}, $\mathbb{A}=-z+E$ is a $C$-asymptotic set. Thus, $E$ can be decomposed into
\begin{equation*}
E=\mathbb{A}+z.
\end{equation*}
Suppose that there are two $C$-asymptotic set $\mathbb{A}_1,\mathbb{A}_2$ and $z_1,z_2\in C$ such that
\begin{equation*}
E=\mathbb{A}_1+z_1=\mathbb{A}_2+z_2,
\end{equation*}
then $\mathbb{A}_1+(z_1-z_2)=\mathbb{A}_2$. Since $\mathbb{A}_1$ is a $C$-asymptotic set, one has
\begin{equation*}
\lim_{x\in\partial\mathbb{A}_1,|x|\rightarrow+\infty}d(x,\partial C)=0,
\end{equation*}
which induces
\begin{equation*}
\lim_{x\in\partial\mathbb{A}_2,|x|\rightarrow+\infty}d(x,(z_1-z_2)+\partial C)=0.
\end{equation*}
If $z_1-z_2\neq o$, then the above formula contradicts with the fact that $\mathbb{A}_2$ is a $C$-asymptotic set. Thus, there are $z_1=z_2$ and $\mathbb{A}_1=\mathbb{A}_2$. This shows $(iii)$ holds.

Conversely, if $(iii)$ holds, then $h_E(v)=h_{z+\mathbb{A}}(v)=h_{\mathbb{A}}(v)+\langle z,v\rangle$ for all $v\in\partial\Omega_{C^\circ}$. For $C$-asymptotic set $\mathbb{A}$, we claim that $h_{\mathbb{A}}(v)=0$ for all $v\in\partial\Omega_{C^\circ}$. Otherwise, if $h_{\mathbb{A}}(v')<0$ for some $v'\in\partial\Omega_{C^\circ}$, then
\begin{equation*}
\mathbb{A}\subset C\cap H^-(v',h_{\mathbb{A}}(v')).
\end{equation*}
Note that $C=[C\cap H^-(v',h_{\mathbb{A}}(v'))]\cup[C\cap H^+(v',h_{\mathbb{A}}(v'))]$, so the above formula contradicts with the formula \ref{limin-formula-C-asymptotic-set}. Thus, $h_E(v)=\langle z,v\rangle$ for all $v\in\partial\Omega_{C^\circ}$, i.e., $(i)$ holds.
\end{proof}

As a special case, we have the following corollary.
\begin{corollary}\label{C-asymptotic-set-C-starting-point-support-function}
Let $E$ be a $C$-pseudo-cone. Then the following statements are equivalent: \\
$(i)$ $E$ is a $C$-asymptotic set;\\
$(ii)$ The origin $o$ is the $C$-starting point of $E$;\\
$(iii)$ $h_E(v)=0$ for all $v\in\partial\Omega_{C^\circ}$.
\end{corollary}

\begin{remark}
Recently, the equivalence between $(i)$ and $(iii)$ was also established by Semenov and Zhao \cite{Semenov-Zhao}.
\end{remark}

\begin{definition}
Let $E$ be a $C$-pseudo-cone. If $E$ satisfies $(i)$ or $(ii)$, $(iii)$ in Theorem \ref{Theorem-the-characteristic-of-C-pseudo-cone-extension}, then $E$ is called a non-degenerated $C$-pseudo-cone; If $E$ does not satisfy $(i)$ in Theorem \ref{Theorem-the-characteristic-of-C-pseudo-cone-extension}, then $E$ is called a degenerated $C$-pseudo-cone.
\end{definition}

\begin{remark}
From the perspective of the $C$-starting point, a non-degenerated $C$-pseudo-cone has its $C$-starting point, but a degenerated $C$-pseudo-cone does have the $C$-starting point. In the plane $\mathbb{R}^2$, each $C$-pseudo-cone all are non-degenerated. However, there are many degenerated $C$-pseudo-cones in high dimension. For example, please see Figure \ref{Fig:degenerated-pseudo-cone}.
\end{remark}

 \begin{figure}[htbp]
 \centering
\begin{minipage}{.6\textwidth}
  \centering
  \includegraphics[width=8cm]{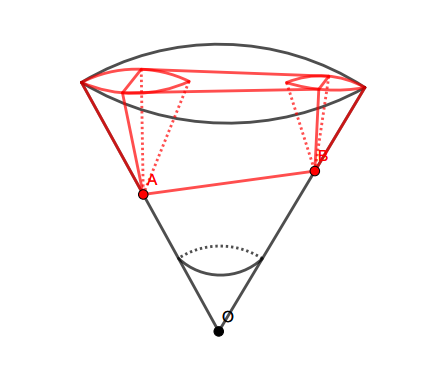}
  \caption{The degenerated $C$-pseudo-cone in $\mathbb{R}^3$}
  \label{Fig:degenerated-pseudo-cone}
\end{minipage}
\end{figure}

This is an example of the degenerated $C$-pseudo-cone in $\mathbb{R}^3$ and was provided by R. Schneider. In which, we place two points $A$ and $B$ on the boundary of the circular cone $C$ whose vertex is $O$. Then the highlighted domain (in red) swept by the motion from $A+C$ to $B+C$ is a pseudo-cone, but it cannot be decomposed as Theorem \ref{Theorem-the-characteristic-of-C-pseudo-cone}.

\begin{lemma}\label{Minkowski-sum-of-C-asymptotic-set}
Let $\mathbb{A}_1,\mathbb{A}_2$ be two $C$-asymptotic sets. Then the Minkowski sum $\mathbb{A}_1+\mathbb{A}_2$ is also a $C$-asymptotic set.
\end{lemma}
\begin{proof}
By Lemma \ref{C-asymptotic-set}, we have
\begin{equation*}
\mathbb{A}_1=\bigcap_{v\in\Omega_{C^\circ}}H^-(v,h_{\mathbb{A}_1}(v)),\,\mathbb{A}_2=\bigcap_{v\in\Omega_{C^\circ}}
H^-(v,h_{\mathbb{A}_2}(v)).
\end{equation*}
Clearly, $\mathbb{A}_1$ and $\mathbb{A}_2$ are unbounded closed convex sets contained in $C$, so the Minkowski sum $\mathbb{A}_1+\mathbb{A}_2$ is also an unbounded closed convex sets in $C$. For any $v\in\mathbb{S}^{n-1}\setminus C^\circ$, we have
\begin{equation*}
\sup\{\langle x+y,v\rangle\,|\,x\in\mathbb{A}_1,y\in\mathbb{A}_2\}
=\sup\{\langle x,v\rangle\,|\,x\in\mathbb{A}_1\}+\sup\{\langle y,v\rangle\,|\,y\in\mathbb{A}_2\}=+\infty.
\end{equation*}
Hence, according to the closeness and convexity of $\mathbb{A}_1+\mathbb{A}_2$, it follows that
\begin{equation}\label{Minkowski-sum-of-C-asymptotic-set-formula-1}
\mathbb{A}_1+\mathbb{A}_2=\bigcap_{v\in\Omega_{C^\circ}\cup\partial\Omega_{C^\circ}}H^-(v,h_{\mathbb{A}_1+\mathbb{A}_2}(v)).
\end{equation}
By Corollary \ref{C-asymptotic-set-C-starting-point-support-function}, we have $h_{\mathbb{A}_1+\mathbb{A}_2}(v)=h_{\mathbb{A}_1}(v)+h_{\mathbb{A}_2}(v)=0$ for $v\in\partial\Omega_{C^\circ}$. Thus,
\begin{equation}\label{Minkowski-sum-of-C-asymptotic-set-formula-2}
\bigcap_{v\in\partial\Omega_{C^\circ}}H^-(v,h_{\mathbb{A}_1+\mathbb{A}_2}(v))=\bigcap_{v\in\partial\Omega_{C^\circ}}H^-(v,0)=C.
\end{equation}
Combining \eqref{Minkowski-sum-of-C-asymptotic-set-formula-1} with \eqref{Minkowski-sum-of-C-asymptotic-set-formula-2}, one has
\begin{equation*}
\begin{aligned}
\mathbb{A}_1+\mathbb{A}_2&=\Bigg(\bigcap_{v\in\Omega_{C^\circ}}H^-(v,h_{\mathbb{A}_1+\mathbb{A}_2}(v))\Bigg)\cap\Bigg(\bigcap_{v\in
\partial\Omega_{C^\circ}}H^-(v,0)\Bigg)\\
&=C\cap\Bigg(\bigcap_{v\in\Omega_{C^\circ}}H^-(v,h_{\mathbb{A}_1+\mathbb{A}_2}(v))\Bigg),
\end{aligned}
\end{equation*}
which implies that $\text{conv}(\mathbb{A}_1+\mathbb{A}_2,C)\subset\mathbb{A}_1+\mathbb{A}_2$ by Definition \ref{Def-C-closed-convex-hull}. Therefore, $\mathbb{A}_1+\mathbb{A}_2$ is a $C$-compatible set. Since $h_{\mathbb{A}_1+\mathbb{A}_2}(v)=0$ for all $v\in\partial\Omega_{C^\circ}$, we conclude that $\mathbb{A}_1+\mathbb{A}_2$ is a $C$-asymptotic set by Corollary \ref{C-asymptotic-set-C-starting-point-support-function}.
\end{proof}
\begin{corollary}\label{Minkowski-sum-of-non-degenerated-C-pseudo-cone}
The Minkowski sum $E+F$ of two non-degenerated $C$-pseudo-cones $E$ and $F$ is still a non-degenerated $C$-pseudo-cone.
\end{corollary}
\begin{proof}
Let $z_1$ and $z_2$ be the $C$-starting point of $E$ and $F$, respectively. By Theorem \ref{Theorem-the-characteristic-of-C-pseudo-cone-extension}, there exists two $C$-asymptotic set $\mathbb{A}_1$ and $\mathbb{A}_2$ such that
\begin{equation*}
E=\mathbb{A}_1+z_1,\,F=\mathbb{A}_2+z_2.
\end{equation*}
By Lemma \ref{Minkowski-sum-of-C-asymptotic-set}, the Minkowski sum $\mathbb{A}_1+\mathbb{A}_2$ is also a $C$-asymptotic set. Noting that $z_1+z_2\in C$ and applying Theorem \ref{Theorem-the-characteristic-of-C-pseudo-cone-extension} again, the set
\begin{equation*}
E+F=(\mathbb{A}_1+z_1)+(\mathbb{A}_2+z_2)=(\mathbb{A}_1+\mathbb{A}_2)+(z_1+z_2)
\end{equation*}
is also a non-degenerated $C$-pseudo-cone.
\end{proof}

Recently, Lemma \ref{Minkowski-sum-of-C-asymptotic-set} was also established by Semenov and Zhao \cite{Semenov-Zhao}. Moreover, they also established the following results.
\begin{lemma}[see \cite{Semenov-Zhao}]\label{Minkowski-sum-of-C-pseudo-cone}
The Minkowski sum $E+F$ of two $C$-pseudo-cones $E$ and $F$ is still a $C$-pseudo-cone.
\end{lemma}

\section{Finiteness of the asymptotic weighted co-volume}\label{section-finiteness}

Recall that the weight function $\Theta: C\setminus\{o\}\rightarrow (0,\infty)$ is a $(-q)$-homogeneous continuous function with $q\in \mathbb{R}$. By the compactness of the set $C(1)=\{x\in C\,|\,\langle x,\mathfrak{u}\rangle=1\}$, we can define two positive constants
\begin{equation*}
m_{\Theta}=\min_{x\in C(1)}\Theta(x),\  \,M_{\Theta}=\max_{x\in C(1)}\Theta(x).
\end{equation*}
For any $x\in C\setminus\{o\}$, since $\frac{1}{\langle x,\mathfrak{u}\rangle}x\in C(1)$, it follows that
\begin{equation}\label{upper-lower-estimate-of-weight}
m_{\Theta}\langle x,\mathfrak{u}\rangle^{-q}\leqslant\Theta(x)\leqslant M_{\Theta}\langle x,\mathfrak{u}\rangle^{-q}.
\end{equation}
Let $x_0\in\text{int}\,E$ and $t_0=\langle x_0,\mathfrak{u}\rangle$. Then, we have $x_0+C=\text{conv}(x_0,C)\subset\text{conv}(E,C)=E$. Similar to \cite{Schneider-A_weighted_Minkowski_theorem}, we define the sequence $t_i=t_0+i$, for $i\in\mathbb{N}$, and the sequences of sets
\begin{equation}\label{sequence}
\left\{
\begin{aligned}
E_i&=\{x\in E\,|\,t_i\leqslant\langle x,\mathfrak{u}\rangle\leqslant t_{i+1}\},\ \partial^*E_i=\{x\in\partial E_i\,|\,t_i<\langle x,\mathfrak{u}\rangle<t_{i+1}\},\\
\overline{E_i}&=E\cap C(t_{i+1})+\{\lambda\mathfrak{u}\,|\,-1\leqslant\lambda\leqslant0\},\ \partial^*\overline{E_i}= \{x\in\partial\overline{E_i}\,|\,t_i<\langle x,\mathfrak{u}\rangle<t_{i+1}\},\\
\underline{E_i}&=E\cap C(t_i)+\{\lambda\mathfrak{u}\,|\,0\leqslant\lambda\leqslant1\},\ \partial^*\underline{E_i}= \{x\in\partial\underline{E_i}\,|\,t_i<\langle x,\mathfrak{u}\rangle<t_{i+1}\},
\end{aligned}
\right.
\end{equation}
where the upper cylinders $\overline{E_i}$, lower cylinders $\underline{E_i}$, and convex bodies $E_i$ satisfy the relationship:
\begin{equation}\label{upper-cylinders-lower-cylinders-convex-bodies-contain}
\underline{E_i}\subset E_i\subset\overline{E_i}.
\end{equation}

Firstly, we present the following result regrading the weighted volume of $C$-pseudo-cones:
\begin{lemma}\label{Finiteness-weighted volume}
Let $E$ be a $C$-pseudo-cone. If $q>n$, then $V_{\Theta}(E)$ is finite; if $q\leqslant n$, then $V_{\Theta}(E)$ is infinite.
\end{lemma}
\begin{proof}
Since $\Theta$ is continuous on $C\setminus\{o\}$ and $o\notin E$, it follows that
\begin{equation*}
\int_{E\cap C^-(t_0)}\Theta(x)\,d\mathcal{H}^n(x)<\infty.
\end{equation*}
Then, the following integral can be divided as
\begin{equation*}
\int_{E\setminus C^-(t_0)}\Theta(x)\,d\mathcal{H}^n(x)=\sum_{i=0}^{+\infty}\int_{E_i}\Theta(x)\,d\mathcal{H}^n(x)\triangleq
\sum_{i=0}^{+\infty}J_i(E).
\end{equation*}
$(i)$. For $q>n$, by \eqref{upper-lower-estimate-of-weight} and \eqref{upper-cylinders-lower-cylinders-convex-bodies-contain}, we have the following estimate:
\begin{equation*}
\begin{aligned}
J_i(E)&\leqslant\int_{E_i}M_{\Theta}\langle x,\mathfrak{u}\rangle^{-q}\,d\mathcal{H}^n(x)\\
&\leqslant M_{\Theta}\int_{E_i}t_i^{-q}\,d\mathcal{H}^n(x)=M_{\Theta}t_i^{-q}\mathcal{H}^n(E_i)\\
&\leqslant M_{\Theta}t_i^{-q}\mathcal{H}^n(\overline{E_i})=M_{\Theta}t_i^{-q}\mathcal{H}^{n-1}(E\cap C(t_{i+1}))\\
&\leqslant M_{\Theta}t_i^{-q}\mathcal{H}^{n-1}(C(t_{i+1}))=M_{\Theta}\mathcal{H}^{n-1}(C(1))\,t_i^{-q}(t_0+i+1)^{n-1},
\end{aligned}
\end{equation*}
which implies
\begin{equation*}
\begin{aligned}
\lim_{i\rightarrow+\infty}i^{q-n+1}J_i(E)&\leqslant\lim_{i\rightarrow+\infty}i^{q-n+1}M_{\Theta}\mathcal{H}^{n-1}(C(1))
t_i^{-q}(t_0+i+1)^{n-1}\\
&=M_{\Theta}\mathcal{H}^{n-1}(C(1))\lim_{i\rightarrow+\infty}\frac{i^q(t_0+i+1)^{n-1}}{i^{n-1}(t_0+i)^q}\\
&=M_{\Theta}\mathcal{H}^{n-1}(C(1))<+\infty.
\end{aligned}
\end{equation*}
According to the comparison test for $p$-series, and since $q-n+1>1$, the infinite series $\sum_{i=0}^{+\infty}J_i(E)$ is convergent, so $V_{\Theta}(E)$ is finite.\\
$(ii)$. For $q\leqslant n$, as above, we have
\begin{equation*}
\begin{aligned}
J_i(E)&\geqslant m_{\Theta}\int_{E_i}t_{i+1}^{-q}\,d\mathcal{H}^n(x)=m_{\Theta}t_{i+1}^{-q}\mathcal{H}^n(E_i)\\
&\geqslant m_{\Theta}t_{i+1}^{-q}\mathcal{H}^n(\underline{E_i})
=m_{\Theta}t_{i+1}^{-q}\mathcal{H}^{n-1}(E\cap C(t_i))\\
&\geqslant m_{\Theta}t_{i+1}^{-q}\mathcal{H}^{n-1}((x_0+C)\cap C(t_i))
=m_{\Theta}\mathcal{H}^{n-1}(x_0+C(1))\,i^{n-1}t_{i+1}^{-q}\\
&=m_{\Theta}\mathcal{H}^{n-1}(C(1))\,i^{n-1}t_{i+1}^{-q},
\end{aligned}
\end{equation*}
which also leads to
\begin{equation*}
\begin{aligned}
\lim_{i\rightarrow+\infty}i^{q-n+1}J_i(E)&\geqslant m_{\Theta}\mathcal{H}^{n-1}(C(1))\lim_{i\rightarrow+\infty}
\frac{i^q}{(t_0+i+1)^q}=m_{\Theta}\mathcal{H}^{n-1}(C(1))>0.
\end{aligned}
\end{equation*}
Again, by the comparison test for $p$-series, and since $q-n+1\leqslant1$, we have $\sum_{i=0}^{+\infty}J_i(E)=+\infty$, so $V_{\Theta}(E)$ is infinite.
\end{proof}

Next, we provide the proof of a key estimate concerning the finiteness of weighted surface area measure, as discussed in \cite[p.\,9]{Schneider-A_weighted_Minkowski_theorem}.
\begin{lemma}
For the sets in \eqref{sequence}, there exists a constant $c$, independent of $i\in \mathbb{N}$, such that
\begin{equation*}
\mathcal{H}^{n-1}((\overline{E_i}\setminus E_i)\cap C(t_i))\leqslant c\mathcal{H}^{n-2}(\partial E\cap C(t_i)).
\end{equation*}
\end{lemma}
\begin{proof}
Firstly, by the definitions in \eqref{sequence}, we have
\begin{equation*}
\begin{aligned}
\mathcal{H}^{n-1}((\overline{E_i}\setminus E_i)\cap C(t_i))&\leqslant\mathcal{H}^{n-1}((\overline{E_i}\setminus E_i)\cap H(\mathfrak{u},t_i))\\
&=\mathcal{H}^{n-1}(\overline{E_i}\cap H(\mathfrak{u},t_i))-\mathcal{H}^{n-1}(E_i\cap H(\mathfrak{u},t_i))\\
&=\mathcal{H}^{n-1}(E\cap C(t_{i+1}))-\mathcal{H}^{n-1}(E\cap C(t_i))\\
&\leqslant\mathcal{H}^{n-1}(C(t_{i+1}))-\mathcal{H}^{n-1}((x_0+C)\cap C(t_i))\\
&=(t_0+i+1)^{n-1}\mathcal{H}^{n-1}(C(1))-i^{n-1}\mathcal{H}^{n-1}(x_0+C(1))\\
&=\mathcal{H}^{n-1}(C(1))[(t_0+i+1)^{n-1}-i^{n-1}]\\
&=\mathcal{H}^{n-1}(C(1))\sum_{k=1}^{n-1}\dbinom{n-1}{k}i^{n-1-k}(t_0+1)^k.\\
\end{aligned}
\end{equation*}
Since the surface area of convex bodies increases monotonically with respect to set inclusion, we have
\begin{equation*}
\mathcal{H}^{n-2}(\partial E\cap C(t_i))\geqslant\mathcal{H}^{n-2}((x_0+\partial C)\cap C(t_i))=i^{n-2}\mathcal{H}^{n-2}(\partial C\cap C(1)),
\end{equation*}
thus, we have
\begin{equation*}
\begin{aligned}
\frac{\mathcal{H}^{n-1}((\overline{E_i}\setminus E_i)\cap C(t_i))}{\mathcal{H}^{n-2}(\partial E\cap C(t_i))}&\leqslant
\frac{\mathcal{H}^{n-1}(C(1))}{\mathcal{H}^{n-2}(\partial C\cap C(1))}\sum_{k=1}^{n-1}\dbinom{n-1}{k}i^{1-k}(t_0+1)^k\\
&\leqslant\frac{\mathcal{H}^{n-1}(C(1))}{\mathcal{H}^{n-2}(\partial C\cap C(1))}\sum_{k=1}^{n-1}\dbinom{n-1}{k}(t_0+1)^k\triangleq c.
\end{aligned}
\end{equation*}
\end{proof}
\begin{remark}\label{Finiteness-remark-weighted-surface-area-measure}
Using the above estimate, Schneider \cite{Schneider-A_weighted_Minkowski_theorem} proved that $S^{\Theta}_{n-1}(E,\cdot)$ is finite for any $q>n-1$ and every $C$-pseudo-cone $E$. In fact, one can also obtain the same results by using the binomial expansion and the comparison test for $p$-series directly. When $q<n-1$, Schneider \cite{Schneider-A_weighted_Minkowski_theorem} provided a counterexample to show that $S^{\Theta}_{n-1}(E,\Omega_{C^\circ})=+\infty$ for some $C$-pseudo-cone $E$.
\end{remark}

\begin{counterexample}
For the critical case $q=n-1$, we provide a counterexample as follows. In the plane $\mathbb{R}^2$, let $\mathfrak{u}=(\frac{\sqrt{2}}{2},\frac{\sqrt{2}}{2})$ and consider the weight function
\begin{equation*}
\Theta(x,y)=\langle\mathfrak{u},(x,y)\rangle^{-1}=\frac{\sqrt{2}}{x+y},\,(x,y)\in\mathbb{R}^2.
\end{equation*}
Choose a fixed cone $C=\{(x,y)\in\mathbb{R}^2\,|\,x\geqslant0,y\geqslant0\}$ and a $C$-pseudo-cone
\begin{equation*}
E=\{(x,y)\in C\,|\,y\geqslant\frac{1}{x}\},
\end{equation*}
then we have
\begin{equation*}
\begin{aligned}
S_{n-1}^\Theta(E,\Omega_{C^\circ})&=\int_{\{(x,y)\in C\,|\,y=\frac{1}{x}\}}\Theta(x,y)\,ds
=\int_0^{+\infty}\frac{\sqrt{2}}{x+\frac{1}{x}}\sqrt{1+\frac{1}{x^2}}\,dx\\
&=\sqrt{2}\int_0^{+\infty}\frac{1}{\sqrt{1+x^2}}\,dx>\sqrt{2}\int_1^{+\infty}\frac{1}{\sqrt{1+x^2}}\,dx\\
&\geqslant\int_1^{+\infty}\frac{1}{x}\,dx=+\infty.
\end{aligned}
\end{equation*}
Therefore, if $q\leqslant n-1$, the weighted surface area measure of a $C$-pseudo-cone may be infinite. However, it may also be finite, as in the case of $C$-full sets.
\end{counterexample}

The weighted co-volume of the $C$-pseudo-cone $E$ can be divided into
\begin{equation*}
\overline{V}_{\Theta}(E)=\int_{(C\setminus E)\cap\,\mathbb{B}^n}\Theta(x)\,d\mathcal{H}^n(x)+
\int_{(C\setminus E)\setminus\mathbb{B}^n}\Theta(x)\,d\mathcal{H}^n(x)\triangleq I^{\Theta}_0(E)+I^{\Theta}_{\infty}(E).
\end{equation*}
It has been proven that $I^{\Theta}_0(E)$ is finite for $q<n$ and $I^{\Theta}_{\infty}(E)$ is finite for $q>n-1$ (see \cite[Lemma 7]{Schneider-A_weighted_Minkowski_theorem}). Thus, $\overline{V}_{\Theta}(E)$ is finite for $n-1<q<n$. Here, we present some additional conclusions on the weighted co-volume.
\begin{lemma}\label{Finiteness-The-flaw-integral-I0E-infinite}
If $q\geqslant n$, then $I^{\Theta}_0(E)$ is infinite for every $C$-pseudo-cone $E$.
\end{lemma}
\begin{proof}
According to \eqref{upper-lower-estimate-of-weight} and the Cauchy-Schwarz inequality, we have
\begin{equation*}
\begin{aligned}
I^{\Theta}_0(E)&=\int_{(C\setminus E)\cap B^n_2}\Theta(x)\,d\mathcal{H}^n(x)\\
&\geqslant m_{\Theta}\int_{(C\setminus E)\cap B^n_2}\langle x,\mathfrak{u}\rangle^{-q}\,d\mathcal{H}^n(x)\\
&\geqslant m_{\Theta}\int_{(C\setminus E)\cap B^n_2}|x|^{-q}\,d\mathcal{H}^n(x).
\end{aligned}
\end{equation*}
Let $r=\min\{1,\min_{x\in E}|x|\}>0$, then $C\cap\,rB^n_2\subset(C\setminus E)\cap B^n_2$. Thus,
\begin{equation*}
\begin{aligned}
I^{\Theta}_0(E)&\geqslant m_{\Theta}\int_{C\cap\,rB^n_2}|x|^{-q}\,d\mathcal{H}^n(x)\\
&=m_{\Theta}\int_{C\cap\,\mathbb{S}^{n-1}}\int_0^r\frac{1}{r^{q-n+1}}\,drdu\\
&=m_{\Theta}\mathcal{H}^{n-1}(C\cap\mathbb{S}^{n-1})\int_0^r\frac{1}{r^{q-n+1}}\,dr.
\end{aligned}
\end{equation*}
Since $q-n+1\geqslant 1$, we have
\begin{equation*}
\int_0^r\frac{1}{r^{q-n+1}}\,dr=+\infty.
\end{equation*}
Thus, $I^{\Theta}_0(E)=+\infty$.
\end{proof}
\begin{counterexample}\label{Finiteness-remark-I-infty-integral}
If $q\leqslant n-1$, there is a counterexample that shows $I^{\Theta}_{\infty}(E)$ is infinite for some $C$-pseudo-cone $E$.
\end{counterexample}
\begin{proof}
In the plane $\mathbb{R}^2$, let us choose a fixed cone $C=\{(x,y)\in\mathbb{R}^2\,|\,x\geqslant0,y\geqslant0\}$ and consider the weight function
\begin{equation*}
\Theta(x,y)=\big\langle(\frac{\sqrt{2}}{2},\frac{\sqrt{2}}{2}),(x,y)\big\rangle^{-q}
=\Big(\frac{\sqrt{2}}{x+y}\Big)^q,\ \,(x,y)\in C,
\end{equation*}
where $0<q\leqslant1$. We construct a $C$-pseudo-cone $E=(1,1)+C=\{(x,y)\in\mathbb{R}^2\,|\,x\geqslant1,y\geqslant1\}$.
For $0<q<1$, we have
\begin{equation*}
\begin{aligned}
I^{\Theta}_{\infty}(E)&=\int_{(C\setminus E)\setminus\mathbb{B}^n}\Theta(x,y)\,d\mathcal{H}^2(x,y)
>\int_{[1,+\infty)\times[0,1]}\Big(\frac{\sqrt{2}}{x+y}\Big)^q\,dxdy\\
&=2^\frac{q}{2}\int_1^{+\infty}\bigg(\int_0^1\frac{1}{(x+y)^q}\,dy\bigg)\,dx
=\frac{2^\frac{q}{2}}{1-q}\int_1^{+\infty}\big((1+x)^{1-q}-x^{1-q}\big)\,dx.
\end{aligned}
\end{equation*}
Since
\begin{equation*}
\lim_{x\rightarrow+\infty}\frac{(1+\frac{1}{x})^{1-q}-1}{\frac{1}{x}}=1-q\in(0,1),
\end{equation*}
there exists $N>1$ such that
\begin{equation*}
(1+x)^{1-q}-x^{1-q}=x^{1-q}\Big((1+\frac{1}{x})^{1-q}-1\Big)\geqslant\frac{1-q}{2}\,x^{-q}
\end{equation*}
for all $x>N$. Thus,
\begin{equation*}
\begin{aligned}
I^{\Theta}_{\infty}(E)&=\frac{2^\frac{q}{2}}{1-q}\int_1^{+\infty}\big((1+x)^{1-q}-x^{1-q}\big)\,dx\\
&\geqslant\frac{2^\frac{q}{2}}{1-q}\bigg(\int_1^N\big((1+x)^{1-q}-x^{1-q}\big)\,dx+\frac{1-q}{2}\int_N^{+\infty}\frac{1}{x^q}\,dx\bigg)\\
&=+\infty.
\end{aligned}
\end{equation*}
For $q=1$, since
\begin{equation*}
\lim_{x\rightarrow+\infty}\frac{\log(1+\frac{1}{x})}{\frac{1}{x}}=1,
\end{equation*}
there exists a constant $N>1$ such that $\log(1+\frac{1}{x})>\frac{1}{2x}$ for all $x>N$. Thus, we have
\begin{equation*}
\begin{aligned}
I^{\Theta}_{\infty}(E)&>\sqrt{2}\int_1^{+\infty}\bigg(\int_0^1\frac{1}{x+y}\,dy\bigg)\,dx
=\sqrt{2}\int_1^{+\infty}\log(1+\frac{1}{x})\,dx\\
&\geqslant\sqrt{2}\bigg(\int_1^N\log(1+\frac{1}{x})\,dx+\int_N^{+\infty}\frac{1}{2x}\,dx\bigg)=+\infty.
\end{aligned}
\end{equation*}
Therefore, if $q\leqslant n-1$, then $I^{\Theta}_{\infty}(E)$ may be infinite. However, it may also be finite.
\end{proof}

Recall the $p$-improper integral over infinite interval:
\begin{equation*}
\int_1^{+\infty}\frac{1}{x^p}\,dx=\left\{
\begin{aligned}
&\text{finite value},\ &p>1,\\
&+\infty,\ &p\leqslant 1,
\end{aligned}
\right.
\end{equation*}
and the $p$-improper integral with singularities:
\begin{equation*}
\int_0^1\frac{1}{x^p}\,dx=\left\{
\begin{aligned}
&+\infty,\ &p\geqslant1,\\
&\text{finite value},\ &p<1.
\end{aligned}
\right.
\end{equation*}
From the above discussion, it is clear that the weighted volume $V_{\Theta}(E)$ of a $C$-pseudo-cone $E$ is the natural $n$-dimensional generalization of the $p$-improper integral over infinite interval, and the integral $I^{\Theta}_0(E)$ is the natural $n$-dimensional generalization of the $p$-improper integral with singularities. However, neither $V_{\Theta}(E)$ nor $I^{\Theta}_0(E)$ effectively captures the relative information about $E$. Furthermore, the weighted co-volume $\overline{V}_{\Theta}(E)$ calculates some invalid volumes of $E$ form the  perspective of the $C$-starting point of $E$. Therefore, we will focus on studying the asymptotic behavior of the weighted co-volume of $E$.

\begin{definition}
Let $z$ be the $C$-starting point of $C$-pseudo-cone $E$, i.e., $E$ can be decomposed into the sum of a $C$-asymptotic set $\mathbb{A}$ and the $C$-starting point $z$, then we define the asymptotic weighted co-volume of $E$ by
\begin{equation*}
T_{\Theta}(E)=T_{\Theta}(\mathbb{A},z)=V_{\Theta}((z+C)\setminus E)=\int_{(z+C)\setminus E}\Theta(x)\,d\mathcal{H}^n(x).
\end{equation*}
\end{definition}
\begin{remark}
Here we consider this asymptotic weighted co-volume $T_{\Theta}(E)$ as a functional of two variables. Thus, if we fix a $C$-asymptotic set $\mathbb{A}$, then $T_{\Theta}(E)=T_{\Theta}(\mathbb{A},\cdot)$ is precisely a generalized function on the cone $C$.
\end{remark}

It is evident that the weighted volume $V_{\Theta}(E)$ of a $C$-pseudo-cone $E$ serves the $n$-dimensional generalization of the $p$-improper integral over infinite interval. However, the $n$-dimensional generalization of the $p$-improper integral with singularities is more intricate, involving both the weighted co-volume $\overline{V}_{\Theta}(E)$ and the asymptotic weighted co-volume $T_{\Theta}(\mathbb{A},z)$. Moreover, the properties of the asymptotic weighted co-volume are superior to those of the weighted co-volume.

\begin{lemma}\label{Finiteness-asymptotic-weighted-co-volume}
Let $E$ be a $C$-asymptotic set and $q>n-1$. If $z\neq o$, $T_{\Theta}(E,z)$ is finite. If $z=o$, the following holds:
\begin{equation*}
T_{\Theta}(E,o)=\left\{
\begin{aligned}
&\text{finite},\ n-1<q<n,\\
&+\infty,\ q\geqslant n.
\end{aligned}
\right.
\end{equation*}
Thus, the origin $o$ is the unique singularity of $T_{\Theta}(E,\cdot)$ for $q\geqslant n$.
\end{lemma}
\begin{proof}
According to the definition of the asymptotic weighted co-volume, we have
\begin{equation*}
\begin{aligned}
T_{\Theta}(E,z)&=\int_{((z+C)\setminus E)\cap C^-(t_0)}\Theta(x)\,d\mathcal{H}^n(x)+\int_{t_0}^{+\infty}\int_{((z+C)\setminus E)\cap C(t)}\Theta(x)\,d\mathcal{H}^{n-1}(x)\,dt\\
&\triangleq\widetilde{J_1}(E)+\widetilde{J_2}(E).
\end{aligned}
\end{equation*}
$(i)$ If $z\neq o$, since $\Theta$ is continuous on $C\setminus\{o\}$ and $((z+C)\setminus E)\cap C^-(t_0)\subset C\setminus\{o\}$ is compact, the integral $\widetilde{J_1}(E)$ is finite. Similar to Lemma 7 in \cite{Schneider-A_weighted_Minkowski_theorem}, using \eqref{upper-lower-estimate-of-weight}, we obtain
\begin{equation*}
\begin{aligned}
\widetilde{J_2}(E)&\leqslant\int_{t_0}^{+\infty}\int_{((z+C)\setminus E)\cap C(t)}M_{\Theta}t^{-q}\,d\mathcal{H}^{n-1}(x)\,dt\\
&=M_{\Theta}\int_{t_0}^{+\infty}t^{-q}\big(\mathcal{H}^{n-1}((z+C)\cap C(t))-\mathcal{H}^{n-1}(E\cap C(t))\big)\,dt\\
&\leqslant M_{\Theta}\int_{t_0}^{+\infty}t^{-q}\big((t-\langle z,\mathfrak{u}\rangle)^{n-1}\mathcal{H}^{n-1}(C\cap C(1))-(t-t_0)^{n-1}
\mathcal{H}^{n-1}(C\cap C(1))\big)\,dt\\
&=M_{\Theta}\mathcal{H}^{n-1}(C\cap C(1))\int_{t_0}^{+\infty}t^{-q}\big((t-\langle z,\mathfrak{u}\rangle)^{n-1}-(t-t_0)^{n-1}\big)\,dt.
\end{aligned}
\end{equation*}
Next, using the binomial expansion:
\begin{equation*}
\begin{aligned}
&\int_{t_0}^{+\infty}\!\!\!\!\!t^{-q}\big((t-\langle z,\mathfrak{u}\rangle)^{n-1}-(t-t_0)^{n-1}\big)\,dt
\!=\!\sum_{i=1}^{n-1}\dbinom{n-1}{i}\big((-\langle z,\mathfrak{u}\rangle)^i-(-t_0)^i\big)
\int_{t_0}^{+\infty}\!\!\!\!\!\frac{1}{t^{q+1+i-n}}\,dt.
\end{aligned}
\end{equation*}
Since $q>n-1$, we have $q+1+i-n>i\geqslant 1$ for $i=1,\cdots,n-1$, which implies that
\begin{equation*}
\int_{t_0}^{+\infty}\frac{1}{t^{q+1+i-n}}\,dt<+\infty.
\end{equation*}
Thus, $T_{\Theta}(E,z)$ is finite.\\
$(ii)$ If $z=o$, similar to \cite[p.\,13]{Schneider-A_weighted_Minkowski_theorem} and Lemma \ref{Finiteness-The-flaw-integral-I0E-infinite}, for $n-1<q<n$ we have
\begin{equation*}
\begin{aligned}
\widetilde{J_1}(E)&=\int_{(C\setminus E)\cap C^-(t_0)}\Theta(x)\,d\mathcal{H}^n(x)\leqslant M_\Theta c^{-q}\int_{B^n_2(R)}|x|^{-q}\,d\mathcal{H}^n(x)\\
&=n\omega_nM_\Theta c^{-q}\int_0^R\frac{1}{r^{q+1-n}}\,dr<+\infty.
\end{aligned}
\end{equation*}
For $q\geqslant n$, we have
\begin{equation*}
\widetilde{J_1}(E)=\int_{(C\setminus E)\cap C^-(t_0)}\Theta(x)\,d\mathcal{H}^n(x)\geqslant m_{\Theta}\int_{(C\setminus E)\cap C^-(t_0)}|x|^{-q}\,d\mathcal{H}^n(x).
\end{equation*}
Let $r=\min\{t_0,\min_{x\in E}|x|\}>0$, then $C\cap\,rB^n_2\subset(C\setminus E)\cap C^-(t_0)$. Thus,
\begin{equation*}
\widetilde{J_1}(E)\geqslant m_{\Theta}\int_{C\cap\,rB^n_2}|x|^{-q}\,d\mathcal{H}^n(x)=m_{\Theta}\mathcal{H}^{n-1}
(C\cap\mathbb{S}^{n-1})\int_0^r\frac{1}{r^{q-n+1}}\,dr=+\infty.
\end{equation*}
It is easy to check that $\widetilde{J_2}(E)<+\infty$ for $z=o$. Therefore,
\begin{equation*}
T_{\Theta}(E,o)=\left\{
\begin{aligned}
&\text{finite},\ & n-1<q<n,\\
&\text{infinite},\ & q\geqslant n.
\end{aligned}
\right.
\end{equation*}
Moreover, it is not hard to prove that $T_{\Theta}(E,z)$ is continuous, that is, $T_{\Theta}(E_i,z_i)\rightarrow T_{\Theta}(E,z)$ as $(E_i,z_i)\rightarrow(E,z)$, similarly to \cite[p.\,14]{Schneider-A_weighted_Minkowski_theorem}.
\end{proof}
\begin{counterexample}\label{Finiteness-remark-asymptotic-weighted-co-volume}
Let $q\leqslant n-1$ and $z\neq o$. Then there exists some $C$-asymptotic set $E$ such that $T_{\Theta}(E,z)=+\infty$.
\end{counterexample}
\begin{proof}
Choose a fixed cone $C=\{(x,y)\in\mathbb{R}^2\,|\,x\geqslant0,y\geqslant0\}$ and consider the weight function
\begin{equation*}
\Theta(x,y)=\Big(\frac{\sqrt{2}}{x+y}\Big)^q,\,(x,y)\in C,
\end{equation*}
where $0<q\leqslant1$. Let $E=\{(x,y)\in C\,|\,y\geqslant\frac{1}{x}\}$ be a $C$-asymptotic set, and choose a $C$-starting point $z=(1,1)$. For $0<q<1$, we have
\begin{equation*}
\begin{aligned}
T_{\Theta}(E,z)&=\int_{(z+C)\setminus E}\Theta(x,y)\,d\mathcal{H}^2(x,y)>\int_{\{(x,y)\in C\,|\,x\geqslant2,1\leqslant y\leqslant1+\frac{1}{x}\}}\Big(\frac{\sqrt{2}}{x+y}\Big)^q\,dxdy\\
&=2^\frac{q}{2}\int_2^{+\infty}\bigg(\int_1^{1+x}\frac{1}{(x+y)^q}\,dy\bigg)\,dx
=\frac{2^\frac{q}{2}}{1-q}\int_2^{+\infty}\big((1+2x)^{1-q}-(1+x)^{1-q}\big)\,dx.
\end{aligned}
\end{equation*}
According to the limit
\begin{equation*}
\lim_{x\rightarrow+\infty}(1+\frac{x}{1+x})^{1-q}-1=2^{1-q}-1\in(0,1),
\end{equation*}
there exists $N>1$ such that
\begin{equation*}
(1+2x)^{1-q}-(1+x)^{1-q}=(1+x)^{1-q}\Big((1+\frac{x}{1+x})^{1-q}-1\Big)>\frac{2^{1-q}-1}{2}\,x^{1-q}
\end{equation*}
for all $x>N$. Thus,
\begin{equation*}
\begin{aligned}
T_{\Theta}(E,z)&>\frac{2^\frac{q}{2}}{1-q}\int_2^{+\infty}\big((1+2x)^{1-q}-(1+x)^{1-q}\big)\,dx\\
&\geqslant\frac{2^\frac{q}{2}}{1-q}\bigg(\int_2^N\big((1+2x)^{1-q}-(1+x)^{1-q}\big)\,dx
+\frac{2^{1-q}-1}{2}\int_N^{+\infty}x^{1-q}\,dx\bigg)\\
&=+\infty.
\end{aligned}
\end{equation*}
For $q=1$, due to
\begin{equation*}
\lim_{x\rightarrow+\infty}\log(1+\frac{x}{1+x})=\log2,
\end{equation*}
there exists a constant $N>1$ such that $\log(1+\frac{x}{1+x})>\frac{1}{2}\log2$ for all $x>N$. Thus,
\begin{equation*}
\begin{aligned}
T_{\Theta}(E,z)&>\sqrt{2}\int_2^{+\infty}\bigg(\int_1^{1+x}\frac{1}{x+y}\,dy\bigg)\,dx
=\sqrt{2}\int_2^{+\infty}\log(1+\frac{x}{1+x})\,dx\\
&\geqslant\sqrt{2}\bigg(\int_2^N\log(1+\frac{x}{1+x})\,dx+\frac{1}{2}\log2\int_N^{+\infty}dx\bigg)=+\infty.
\end{aligned}
\end{equation*}
Therefore, if $q\leqslant n-1$, the asymptotic weighted co-volume may be infinite, but it can also be finite.
\end{proof}

\begin{proof}[Proof of Theorem \ref{Theorem-finite-table}]
Let $E$ be a $C$-pseudo-cone and let $\mathbb{A}$ be a $C$-asymptotic set with $z\neq o$. According to Lemma \ref{Finiteness-The-flaw-integral-I0E-infinite}, the weighted co-volume $\overline{V}_{\Theta}(E)$ is infinite for $q\geqslant n$; as stated in Remark \ref{Finiteness-remark-I-infty-integral}, $\overline{V}_{\Theta}(E)$ can be either finite or infinite for $0<q\leqslant n-1$. The behavior of $T_{\Theta}(\mathbb{A},o)$ is same as $\overline{V}_{\Theta}(E)$. By Remark \ref{Finiteness-remark-weighted-surface-area-measure}, the weighted surface area measure $S_{n-1}^{\Theta}(E,\Omega_{C^\circ})$ can also be finite or infinite for $0\leqslant q\leqslant n-1$. From Lemma \ref{Finiteness-weighted volume}, the weighted volume $V_{\Theta}(E)$ is finite for $q>n$ and infinite for $0<q\leqslant n$. According to Lemma \ref{Finiteness-asymptotic-weighted-co-volume} and Remark \ref{Finiteness-remark-asymptotic-weighted-co-volume}, $T_{\Theta}(\mathbb{A},z)$ is finite for $q>n-1$, but it may be finite or infinite for $0<q\leqslant n-1$. Finally, combining these results with those in \cite{Schneider-A_weighted_Minkowski_theorem}, we can derive Theorem \ref{Theorem-finite-table}, which summarizes the behavior of the weighted co-volume, surface area measure, and asymptotic weighted co-volume.
\begin{center}
\setlength{\tabcolsep}{1mm}
\renewcommand\arraystretch{1.8}
\begin{tabular}{|l|l|l|l|l|l|}
  \hline
  \diagbox{condition}{measure} & $S_{n-1}^{\Theta}(E,\cdot)$ & $V_{\Theta}(E)$ & $\overline{V}_{\Theta}(E)$ & $T_{\Theta}(\mathbb{A},o)$ & $T_{\Theta}(\mathbb{A},z)$\\  \hline
  $q>n$ & \textcolor{blue}{finite} & \textcolor{blue}{finite} & \textcolor{red}{$+\infty$} & \textcolor{red}{$+\infty$} & \textcolor{blue}{finite} \\  \hline
  $q=n$ & \textcolor{blue}{finite} & $+\infty$ & \textcolor{red}{$+\infty$} & $+\infty$ & \textcolor{blue}{finite} \\  \hline
  $n-1<q<n$ & \textcolor{blue}{finite} & $+\infty$ & \textcolor{blue}{finite} & finite & \textcolor{blue}{finite} \\ \hline
  $0<q\leqslant n-1$ & finite or $+\infty$ & $+\infty$ & finite or $+\infty$ & finite or $+\infty$ & finite or $+\infty$ \\
  \hline
\end{tabular}
\end{center}
In the table above, the blue color indicates that there is some useful and interesting information.
\end{proof}

\section{Integral formulas and convolution formulas}

In Sections \ref{section-asymptotic} and \ref{section-finiteness}, we discussed several notable properties of the asymptotic $C$-pseudo-cone. Based on these properties, we can explore the relationships between $S^{\Theta}_{n-1}(E,\cdot)$, $\overline{V}_{\Theta}(E)$, and $T_{\Theta}(E)$. Firstly, we present the following results.

\begin{lemma}\label{Lemma-Integral-formula-co-volume-surface-area-measure}
Let $E$ be a $C$-pseudo-cone and $0\leqslant q<n$. Then, there exist integral representations of the weighted co-volume of $E$ as follows:

\begin{align}
\overline{V}_{\Theta}(E)&=\frac{1}{n-q}\int_{\Omega_C^e}\Theta(u)\rho^{n-q}_E(u)\,du,
\label{Polar-coordinates-co-volume}\\
\overline{V}_{\Theta}(E)&=\frac{1}{n-q}\int_{\Omega_{C^\circ}^e}\overline{h}_E(v)\,dS^{\Theta}_{n-1}(E,v).
\label{Integral-formula-weighted-surface-area-measure-weighted-co-volume}
\end{align}

\end{lemma}
\begin{proof}
By using polar coordinates and the Newton-Leibniz formula for integrals, we have
\begin{equation*}
\overline{V}_{\Theta}(E)=\int_{\Omega_C^e}\int_0^{\rho_E(u)}\Theta(ru)r^{n-1}\,drdu=\frac{1}{n-q}
\int_{\Omega_C^e}\Theta(u)\rho^{n-q}_E(u)\,du.
\end{equation*}
Applying the co-area formula \eqref{The-co-area-formula-in-convex-geometry}, we obtain
\begin{equation*}
\begin{aligned}
\overline{V}_{\Theta}(E)&=\frac{1}{n-q}\int_{\Omega_C^e}\Theta(u)\rho^{-q}_E(u)\overline{h}_E(\alpha_E(u))\frac{\rho_E^n(u)}{\overline{h}_E(\alpha_E(u))}
\,d\mu_{\mathbb{S}^{n-1}}(u)\\
&=\frac{1}{n-q}\int_{\partial_eE}\Theta\left(\frac{x}{|x|}\right)|x|^{-q}\overline{h}_E(\nu_E(x))\,d\mathcal{H}^{n-1}(x)\\
&=\frac{1}{n-q}\int_{\partial_eE}\overline{h}_E(\nu_E(x))\Theta(x)\,d\mathcal{H}^{n-1}(x),
\end{aligned}
\end{equation*}
where $\overline{h}_E(\alpha_E(u))=\overline{h}_E(\nu_E(x))>0$. By the push-forward formula \eqref{weighted-surface-area-measure-equivalent-formula}, we conclude that
\begin{equation*}
\int_{\partial_eE}\overline{h}_E(\nu_E(x))\Theta(x)\,d\mathcal{H}^{n-1}(x)=
\int_{\Omega_{C^\circ}^e}\overline{h}_E(v)\,dS^{\Theta}_{n-1}(E,v).
\end{equation*}
\end{proof}

\begin{remark}
The integral formula analogous to \eqref{Integral-formula-weighted-surface-area-measure-weighted-co-volume} for the volume functional in the classical case can be established using approximations of convex polytopes (see e.g., \cite[p.\,275-276]{Schneider-book}). The study of the push-forward measure of the radial Gauss image was first undertaken by \cite{Huang-Lutwak-Yang-Zhang-Geometric_measures}, utilizing the Gauss-Green formula for finite perimeter sets and various approximation techniques. In fact, the co-area formula of Federer \cite{Federer-book} serves as a more powerful tool for these problems (see \cite[chapter.~4]{Schneider-book} and \cite[p.~170]{Hug-Weil-book}). Generally, the volume formulas for general measures with continuous density have also been derived using the co-area formula, as demonstrated in \cite{Kryvonos-Langharst-Weighted_Minkowski_existence_theorem_and_projection_bodies}. In the case of $C$-close sets, Schneider \cite{Schneider-A_Brunn_Minkowski_theory} utilized the volume formula for convex bodies in conjunction with approximation methods. Here, we also present a general version of the volume formula \eqref{Integral-formula-weighted-surface-area-measure-weighted-co-volume} utilizing the co-area formula.
\end{remark}

\begin{lemma}\label{Lemma-Integral-formula-asymptotic-co-volume-surface-area-measure}
Suppose $q>n-1$. Let $E$ be a $C$-pseudo-cone with its decomposition $E=z+\mathbb{A}$ for the $C$-asymptotic set $\mathbb{A}$ and the starting point $z\neq o$. If $q\neq n$, we have
\begin{equation}
\begin{aligned}
T_{\Theta}(\mathbb{A},z)&=\frac{1}{n-q}\int_{\Omega_{C^\circ}^e}\overline{h}_E(v)\,dS_{n-1}^\Theta(E,v)\\
&\quad+\frac{1}{n-q}\int_{\partial\Omega_{C^\circ}}\langle z,v\rangle\,dS_{n-1}^\Theta(z+C,v).
\end{aligned}
\end{equation}
If $q=n$, we have
\begin{equation}
\begin{aligned}
T_{\Theta}(\mathbb{A},z)&=\int_{\partial_eE}\Theta(x)\overline{h}_E(\nu_E(x))\log|x|\,d\mathcal{H}^{n-1}(x)\\
&\quad+\int_{z+\partial C}\Theta(x)\langle z,\nu_{z+C}(x)\rangle\log|x|\,d\mathcal{H}^{n-1}(x).
\end{aligned}
\end{equation}
\end{lemma}
\begin{proof}
If $n\neq q>n-1$, since $z\neq o$, we have $\Omega_C^e=\Omega_C$ and $\partial_eE=\partial E$. Thus,
\begin{equation*}
\begin{aligned}
T_{\Theta}(\mathbb{A},z)&=\int_{(z+C)\setminus E}\Theta(x)\,d\mathcal{H}^n(x)
=\int_{\Omega_C^e}\int_{\rho_{z+C}(u)}^{\rho_E(u)}\Theta(ru)r^{n-1}\,drdu\\
&=\int_{\Omega_C^e}\Theta(u)\int_{\rho_{z+C}(u)}^{\rho_E(u)}r^{n-1-q}\,drdu
=\frac{1}{n-q}\int_{\Omega_C^e}\Theta(u)(\rho^{n-q}_E(u)-\rho^{n-q}_{z+C}(u))\,du\\
&=\frac{1}{n-q}\int_{\Omega_C^e}\Theta(u)\rho^{n-q}_E(u)\,du-
\frac{1}{n-q}\int_{\Omega_C}\Theta(u)\rho^{n-q}_{z+C}(u)\,du\\
&=\frac{1}{n-q}\int_{\partial_eE}\Theta(x)\overline{h}_E(\nu_E(x))\,d\mathcal{H}^{n-1}(x)\\
&\quad-\frac{1}{n-q}\int_{z+\partial C}\Theta(x)\overline{h}_{z+C}(\nu_{z+C}(x))\,d\mathcal{H}^{n-1}(x)\\
&=\frac{1}{n-q}\int_{\Omega_{C^\circ}^e}\overline{h}_E(v)\,dS_{n-1}^\Theta(E,v)+
\frac{1}{n-q}\int_{\partial\Omega_{C^\circ}}\langle z,v\rangle\,dS_{n-1}^\Theta(z+C,v).
\end{aligned}
\end{equation*}
If $q=n$, we have
\begin{equation*}
\begin{aligned}
T_{\Theta}(\mathbb{A},z)&=\int_{\Omega_C^e}\Theta(u)\int_{\rho_{z+C}(u)}^{\rho_E(u)}\frac{1}{r}\,drdu\\
&=\int_{\Omega_C^e}\Theta(u)(\log\rho_E(u)-\log\rho_{z+C}(u))\,du\\
&=\int_{\Omega_C^e}\Theta(u)\log\rho_E(u)\,du-\int_{\Omega_C}\Theta(u)\log\rho_{z+C}(u)\,du\\
&=\int_{\partial_eE}\Theta(x)\overline{h}_E(\nu_E(x))\log|x|\,d\mathcal{H}^{n-1}(x)\\
&\quad+\int_{z+\partial C}\Theta(x)\langle z,\nu_{z+C}(x)\rangle\log|x|\,d\mathcal{H}^{n-1}(x).
\end{aligned}
\end{equation*}
\end{proof}

According to the definition of the $r$-th dual volume $\widetilde{V}_r(\cdot)$ given in \cite{Li-Ye-Zhu-The_dual_Minkowski_problem}, the formula \eqref{Polar-coordinates-co-volume} implies that $\widetilde{V}_r(\cdot)$ is finite, as follows.

\begin{corollary}
For any $r\in(-\infty,0)\cup(0,1)$, the $r$-th dual volume $\widetilde{V}_r(E)$ is finite for every $C$-pseudo-cone $E$.
\end{corollary}
\begin{proof}
Let $r<0$. Then we have
\begin{equation*}
\widetilde{V}_r(E)=\frac{1}{n}\int_{\Omega_C^e}\rho^r_E(u)\,du
\leqslant\frac{1}{n}\big(\min_{u\in\Omega_C^e}\rho_E(u)\big)^r\int_{\Omega_C^e}du<+\infty.
\end{equation*}
Now, let $r\in(0,1)$ with $r=n-q$, so $q\in(n-1,n)$.
By the estimate \eqref{upper-lower-estimate-of-weight}, the formula \eqref{Polar-coordinates-co-volume}, and the Cauchy-Schwarz inequality, we have
\begin{equation*}
\begin{aligned}
\overline{V}_{\Theta}(E)&=\frac{1}{n-q}\int_{\Omega_C^e}\Theta(u)\rho^{n-q}_E(u)\,du\\
&\geqslant\frac{m_{\Theta}}{n-q}\int_{\Omega_C^e}\langle u,\mathfrak{u}\rangle^{-q}\rho^{n-q}_E(u)\,du\\
&\geqslant\frac{m_{\Theta}}{n-q}\int_{\Omega_C^e}\rho^{n-q}_E(u)\,du.
\end{aligned}
\end{equation*}
Thus, according to the finiteness of the weighted co-volume, we have
\begin{equation*}
\widetilde{V}_r(E)=\widetilde{V}_{n-q}(E)=\frac{1}{n}\int_{\Omega_C^e}\rho^{n-q}_E(u)\,du
\leqslant\frac{n-q}{n\,m_{\Theta}}\overline{V}_{\Theta}(E)<+\infty.
\end{equation*}
\end{proof}

\begin{lemma}
Let $E$ be a $C$-pseudo-cone with its decomposition $E=z+\mathbb{A}$ for the $C$-asymptotic set $\mathbb{A}$ and the starting point $z\neq o$. If $q>0$, we have the following convolution formula:
\begin{equation}\label{The-convolution-formula}
T_{\Theta}(\mathbb{A},z)+V_{\Theta}(E)=\boldsymbol{\chi}_{-C}\ast\Theta\,(z).
\end{equation}
In particular, if $q>n$, we have
\begin{equation}\label{The-convolution-formula-q-big-than-n}
T_{\Theta}(\mathbb{A},z)=\boldsymbol{\chi}_{-C}\ast\Theta\,(z)-V_{\Theta}(z+\mathbb{A}),\,z\in C.
\end{equation}
\end{lemma}
\begin{proof}
Suppose $q>n$. For $z\neq o$, both $V_{\Theta}(E)$ and $T_{\Theta}(\mathbb{A},z)$ are finite. Noting that
\begin{equation*}
\boldsymbol{\chi}_{z+C}(x)=\boldsymbol{\chi}_C(x-z)=\boldsymbol{\chi}_{-C}(z-x),\,x\in\mathbb{R}^n,
\end{equation*}
we have
\begin{equation*}
\begin{aligned}
T_{\Theta}(\mathbb{A},z)&=\int_{\mathbb{R}^n}\boldsymbol{\chi}_{(z+C)\setminus E}(x)\,\Theta(x)\,d\mathcal{H}^n(x)\\
&=\int_{\mathbb{R}^n}\big(\boldsymbol{\chi}_{z+C}(x)-\boldsymbol{\chi}_E(x)\big)\Theta(x)\,d\mathcal{H}^n(x)\\
&=\int_{\mathbb{R}^n}\boldsymbol{\chi}_{z+C}(x)\Theta(x)\,d\mathcal{H}^n(x)-\int_{\mathbb{R}^n}\boldsymbol{\chi}_E(x)\Theta(x)
\,d\mathcal{H}^n(x)\\
&=\int_{\mathbb{R}^n}\boldsymbol{\chi}_{-C}(z-x)\Theta(x)\,d\mathcal{H}^n(x)-\int_E\Theta(x)\,d\mathcal{H}^n(x)\\
&=\boldsymbol{\chi}_{-C}\ast\Theta\,(z)-V_{\Theta}(E),
\end{aligned}
\end{equation*}
where the weight function $\Theta$ can be extended to the whole space $\mathbb{R}^n$, making the convolution $\boldsymbol{\chi}_{-C}\ast\Theta$ well-defined.

For $z=o$, we have $T_{\Theta}(\mathbb{A},o)=+\infty$ and
\begin{equation*}
\boldsymbol{\chi}_{-C}\ast\Theta\,(o)=\int_C\Theta(x)\,d\mathcal{H}^n(x)>\int_{C\cap B^n_2}\Theta(x)\,d\mathcal{H}^n(x)=+\infty.
\end{equation*}
Thus, \eqref{The-convolution-formula-q-big-than-n} holds. Since $V_{\Theta}(E)$ is finite, \eqref{The-convolution-formula} also holds for $q>n$.

Now, suppose $q\leqslant n$. In this case, $V_{\Theta}(E)=+\infty$. Since $z+C$ is also a $C$-pseudo-cone, we have
\begin{equation*}
\boldsymbol{\chi}_{-C}\ast\Theta\,(z)=\int_{z+C}\Theta(x)\,d\mathcal{H}^n(x)=V_{\Theta}(z+C)=+\infty.
\end{equation*}
Thus, \eqref{The-convolution-formula} holds trivially.
\end{proof}

It is easy to verify the following property from the definition of $T_\Theta$, we omit the proof.
\begin{lemma}
Let $E$ be a $C$-pseudo-cone with its decomposition $E=z+\mathbb{A}$, where $\mathbb{A}$ is the $C$-asymptotic set and $z\neq o$. If $q>n$, then for any $t>0$, we have
\begin{equation*}
T_\Theta(t\mathbb{A},tz)=t^{n-q}T_\Theta(\mathbb{A},z).
\end{equation*}
\end{lemma}

If $\Theta$ is $C^1$-smooth function on $C\setminus\{o\}$, then we have the following variational results:
\begin{lemma}
Let $E$ be a $C$-pseudo-cone with its decomposition $E=z+\mathbb{A}$, where $\mathbb{A}$ is the $C$-asymptotic set and $z\neq o$. If $q>n$ and $\Theta$ is $C^1$-smooth function on $C\setminus\{o\}$, then for any $t>0$, we have
\begin{equation}\label{Deriative-fornula-of-asymptotic-co-volume}
t^{q+1-n}\frac{d\,T_\Theta(\mathbb{A},tz)}{dt}=(n-q)\boldsymbol{\chi}_{-C}\ast\Theta\,(z)
-\int_{z+\frac{1}{t}\mathbb{A}}\langle\nabla\Theta(x),z\rangle\,d\mathcal{H}^n(x).
\end{equation}
In particular,
\begin{equation*}
\frac{d}{dt}\bigg|_{t=1}T_\Theta(\mathbb{A},tz)=(n-q)\boldsymbol{\chi}_{-C}\ast\Theta\,(z)
-\int_E\langle\nabla\Theta(x),z\rangle\,d\mathcal{H}^n(x).
\end{equation*}
\end{lemma}
\begin{proof}
Firstly, we observe that
\begin{equation*}
\begin{aligned}
\boldsymbol{\chi}_{-C}\ast\Theta\,(tz)=&\int_{\mathbb{R}^n}\boldsymbol{\chi}_{-C}(x)\Theta(tz-x)
\,d\mathcal{H}^n(x)\\
=&t^{-q}\int_{\mathbb{R}^n}\boldsymbol{\chi}_{-C}(x)\Theta(z-\frac{1}{t}x)\,d\mathcal{H}^n(x)\\
=&t^{n-q}\int_{\mathbb{R}^n}\boldsymbol{\chi}_{-C}(ty)\Theta(z-y)\,d\mathcal{H}^n(y)\\
=&t^{n-q}\int_{\mathbb{R}^n}\boldsymbol{\chi}_{-C}(y)\Theta(z-y)\,d\mathcal{H}^n(y)\\
=&t^{n-q}\,\boldsymbol{\chi}_{-C}\ast\Theta\,(z),
\end{aligned}
\end{equation*}
then
\begin{equation*}
\begin{aligned}
T_{\Theta}(\mathbb{A},tz)&=\boldsymbol{\chi}_{-C}\ast\Theta\,(tz)-V_{\Theta}(tz+\mathbb{A})\\
&=t^{n-q}\,\boldsymbol{\chi}_{-C}\ast\Theta\,(z)-V_{\Theta}(tz+\mathbb{A}).
\end{aligned}
\end{equation*}

On the other hand,
\begin{equation*}
\begin{aligned}
V_{\Theta}(tz+\mathbb{A})&=t^{n-q}V_\Theta(z+\frac{1}{t}\mathbb{A})
=t^{n-q}\int_{z+\frac{1}{t}\mathbb{A}}
\Theta(x)\,d\mathcal{H}^n(x)\\
&=t^{n-q}\int_{\mathbb{A}}\Theta(z+\frac{1}{t}y)\frac{1}{t^n}\,d\mathcal{H}^n(y)\\
&=t^{-q}\int_{\mathbb{A}}
\Theta(z+\frac{1}{t}y)\,d\mathcal{H}^n(y).
\end{aligned}
\end{equation*}
Taking the derivative with respect to $t$, we obtain
\begin{equation*}
\begin{aligned}
\frac{d\,T_\Theta(\mathbb{A},tz)}{dt}&=\frac{d}{dt}\big(t^{n-q}\,\boldsymbol{\chi}_{-C}\ast\Theta\,(z)\big)-
\frac{d}{dt}V_{\Theta}(tz+\mathbb{A})\\
&=(n-q)\boldsymbol{\chi}_{-C}\ast\Theta\,(z)t^{n-q-1}+qt^{-q-1}\int_{\mathbb{A}}\Theta(z+\frac{1}{t}y)\,d\mathcal{H}^n(y)\\
&\quad+t^{-q-1}\int_{\mathbb{A}}\langle\nabla\Theta(z+\frac{1}{t}y),\frac{1}{t}y\rangle\,d\mathcal{H}^n(y)\\
&=(n-q)\boldsymbol{\chi}_{-C}\ast\Theta\,(z)t^{n-q-1}+qt^{n-q-1}\int_{\frac{1}{t}\mathbb{A}}\Theta(z+x)\,d\mathcal{H}^n(x)\\
&\quad+t^{n-q-1}\int_{\frac{1}{t}\mathbb{A}}\langle\nabla\Theta(z+x),x\rangle\,d\mathcal{H}^n(x),
\end{aligned}
\end{equation*}
which implies
\begin{equation*}
t^{q+1-n}\frac{d\,T_\Theta(\mathbb{A},tz)}{dt}=(n-q)\boldsymbol{\chi}_{-C}\ast\Theta\,(z)+\int_{\frac{1}{t}\mathbb{A}}\big(q\Theta(z+x)+
\langle\nabla\Theta(z+x),x\rangle\big)\,d\mathcal{H}^n(x).
\end{equation*}

Since $\Theta$ is a $(-q)$-homogeneous function on $C\setminus\{o\}$, i.e.,
\begin{equation*}
\Theta(tx)=t^{-q}\Theta(x).
\end{equation*}
Differentiating both sides of the above formula with respect to $t$ at $t=1$, we obtain
\begin{equation*}
\langle\nabla\Theta(x),x\rangle=-q\Theta(x),
\end{equation*}
which gives
\begin{equation*}
\begin{aligned}
&\int_{\frac{1}{t}\mathbb{A}}\big(q\Theta(z+x)+\langle\nabla\Theta(z+x),x\rangle\big)\,d\mathcal{H}^n(x)\\
=&\int_{z+\frac{1}{t}\mathbb{A}}\big(q\Theta(y)+\langle\nabla\Theta(y),y-z\rangle\big)\,d\mathcal{H}^n(y)\\
=&-\int_{z+\frac{1}{t}\mathbb{A}}\langle\nabla\Theta(y),z\rangle\,d\mathcal{H}^n(y).
\end{aligned}
\end{equation*}
Thus, we have
\begin{equation*}
t^{q+1-n}\frac{d\,T_\Theta(\mathbb{A},tz)}{dt}=(n-q)\boldsymbol{\chi}_{-C}\ast\Theta\,(z)
-\int_{z+\frac{1}{t}\mathbb{A}}\langle\nabla\Theta(y),z\rangle\,d\mathcal{H}^n(y).
\end{equation*}
Let $t=1$, then
\begin{equation*}
\frac{d}{dt}\bigg|_{t=1}T_\Theta(\mathbb{A},tz)=(n-q)\boldsymbol{\chi}_{-C}\ast\Theta\,(z)
-\int_E\langle\nabla\Theta(y),z\rangle\,d\mathcal{H}^n(y).
\end{equation*}
\end{proof}

Given a fixed $C$-asymptotic set $\mathbb{A}$, $T_{\Theta}(\mathbb{A},\cdot)$ is defined as a generalized function on $C$. We can establish the following decay estimate for $T_{\Theta}(\mathbb{A},\cdot)$ as it approaches infinity.
\begin{lemma}[Decay estimate of the asymptotic weighted co-volume]\label{Decay estimate of the asymptotic weighted co-volume functinal}
Let $E$ be a $C$-pseudo-cone with its decomposition $E=z+\mathbb{A}$, where $\mathbb{A}$ is the $C$-asymptotic set and $z\neq o$. If $q>n$, then there exists a constant $M(C,\Theta,E)$, depending only on $C$, $\Theta$, and $E$, such that
\begin{equation*}
T_\Theta(\mathbb{A},tz)\leqslant M(C,\Theta,E)\,t^{n-q-1},
\end{equation*}
for sufficiently large $t$. Therefore,
\begin{equation*}
T_\Theta(\mathbb{A},tz)=o(t^{n-q})\ \mbox{as}\ t\rightarrow+\infty.
\end{equation*}
\end{lemma}

\begin{proof}
Let $f(t)=T_\Theta(\mathbb{A},tz)$. We choose a point $x_0\in\text{int}\,\mathbb{A}$, and denote $s_0=\langle\mathfrak{u},z\rangle$ and $\bar{s}=\langle\mathfrak{u},x_0\rangle$. Similar to Lemma \ref{Finiteness-asymptotic-weighted-co-volume}, we can derive the following estimate by the Fubini theorem, the formula \eqref{upper-lower-estimate-of-weight} and the binomial expansion theorem:

\begin{align}
0\leqslant f(t)&=\int_{tz+C\setminus\mathbb{A}}\Theta(x)\,d\mathcal{H}^n(x)\nonumber\\
&=\int_{ts_0}^{+\infty}\int_{(tz+C\setminus\mathbb{A})\cap C(s)}\Theta(x)\,d\mathcal{H}^{n-1}(x)\,ds\nonumber\\
&\leqslant M_{\Theta}\int_{ts_0}^{+\infty}s^{-q}\int_{(tz+C\setminus\mathbb{A})\cap C(s)}\,d\mathcal{H}^{n-1}(x)\,ds\nonumber\\
&\leqslant M_{\Theta}\mathcal{H}^{n-1}(C\cap C(1))\int_{ts_0}^{+\infty}s^{-q}\big((s-ts_0)^{n-1}-(s-ts_0-\bar{s})^{n-1}\big)\,ds
\nonumber\\
&=-\sum_{i=1}^{n-1}\dbinom{n-1}{i}M_{\Theta}\mathcal{H}^{n-1}(C\cap C(1))(-\bar{s})^i\int_{ts_0}^{+\infty}\frac{(s-ts_0)^{n-1-i}}{s^q}
\,ds.\label{estimate-of-f-t-infty}
\end{align}

Since $s_0>0$ and $q-n+1+i>1$, the last terms of $(q-n+1+i)$-singular integral converge to zero, i.e.,
\begin{equation*}
\begin{aligned}
0\leqslant\int_{ts_0}^{+\infty}\frac{(s-ts_0)^{n-1-i}}{s^q}\,ds
\leqslant\int_{ts_0}^{+\infty}\frac{1}{s^{q-n+1+i}}\,ds\,\rightarrow0\ \,
\text{as}\ t\rightarrow+\infty.
\end{aligned}
\end{equation*}
By the L'Hopital's rule, we have
\begin{equation*}
\begin{aligned}
\lim_{t\rightarrow+\infty}\frac{1}{t^{n-q-1}}\int_{ts_0}^{+\infty}\frac{1}{s^{q-n+1+i}}\,ds
&=\lim_{t\rightarrow+\infty}\frac{-s_0(ts_0)^{n-q-1-i}}{(n-q-1)t^{n-q-2}}\\
&=\left\{
\begin{aligned}
&0, & i=2,\cdots,n-1,\\
&\frac{s_0^{n-q-1}}{q+1-n}, & i=1.
\end{aligned}
\right.
\end{aligned}
\end{equation*}
Thus, by the Squeeze Theorem, we have

\begin{align}
&\lim_{t\rightarrow+\infty}\frac{1}{t^{n-q-1}}\int_{ts_0}^{+\infty}\frac{(s-ts_0)^{n-1-i}}{s^q}\,ds=0,\,i=2,\cdots,n-1,
\label{squeeze-limit-1}\\
&\limsup_{t\rightarrow+\infty}\frac{1}{t^{n-q-1}}\int_{ts_0}^{+\infty}\frac{(s-ts_0)^{n-2}}{s^q}\,ds\leqslant\frac{s_0^{n-q-1}}{q+1-n}.
\label{squeeze-limit-2}
\end{align}

Using the estimates \eqref{estimate-of-f-t-infty}, \eqref{squeeze-limit-1}, and \eqref{squeeze-limit-2}, and letting $t\rightarrow+\infty$, we have
\begin{equation*}
\begin{aligned}
\limsup_{t\rightarrow+\infty}\frac{f(t)}{t^{n-q-1}}\leqslant\dbinom{n-1}{1}M_{\Theta}\mathcal{H}^{n-1}(C\cap C(1))
\frac{\bar{s}s_0^{n-q-1}}{q+1-n}\triangleq M(C,\Theta,\mathbb{A},z).
\end{aligned}
\end{equation*}
Thus, we conclude that $f(t)\leqslant M(C,\Theta,\mathbb{A},z)t^{n-q-1}$ for sufficiently large $t$.
\end{proof}

Using this decay estimate, we obtain the following imperceptible formula:
\begin{theorem}\label{Formula-convolution-and-driection-derivative-of-Theta}
Let $\Theta$ be $C^1$-smooth function on $C\setminus\{o\}$ with $q>n$. For any $o\neq z\in C$, the following formula holds:
\begin{equation*}
\boldsymbol{\chi}_{-C}\ast\Theta\,(z)=\frac{1}{n-q}\int_{z+C}\frac{\partial\Theta}{\partial z}(x)\,d\mathcal{H}^n(x).
\end{equation*}
\end{theorem}

\begin{proof}
Let $\mathbb{A}$ be a $C$-asymptotic set. We claim that:
\begin{equation*}
\lim_{t\rightarrow+\infty}\int_{z+\frac{1}{t}\mathbb{A}}\langle\nabla\Theta(x),z\rangle\,d\mathcal{H}^n(x)
=\int_{z+C}\langle\nabla\Theta(x),z\rangle\,d\mathcal{H}^n(x).
\end{equation*}
Since the weight function $\Theta:C\setminus\{o\}\rightarrow(0,+\infty)$ is a $C^1$-smooth and $(-q)$-homogeneous function, $|\nabla\Theta(x)|$ is a non-negative continuous and $-(q+1)$-homogeneous function. By Lemma \ref{Finiteness-weighted volume}, we have
\begin{equation*}
\int_{z+C}|\nabla\Theta(x)|\,d\mathcal{H}^n(x)<+\infty,
\end{equation*}
where a non-negative integrable function does not affect this result. According to the absolute value inequality, we have
\begin{equation*}
\begin{aligned}
\bigg|\int_{z+C}\langle\nabla\Theta(x),z\rangle\,d\mathcal{H}^n(x)\bigg|&\leqslant
\int_{z+C}\big|\langle\nabla\Theta(x),z\rangle\big|\,d\mathcal{H}^n(x)\\
&\leqslant|z|\int_{z+C}|\nabla\Theta(x)|\,d\mathcal{H}^n(x)\\
&<+\infty.
\end{aligned}
\end{equation*}
Noting that $|\chi_{z+\frac{1}{t}\mathbb{A}}(x)\langle\nabla\Theta(x),z\rangle|\leqslant|z|
\,|\nabla\Theta(x)|$, by the dominated convergence theorem, we have
\begin{equation*}
\begin{aligned}
\lim_{t\rightarrow+\infty}\int_{z+\frac{1}{t}\mathbb{A}}\langle\nabla\Theta(x),z\rangle\,d\mathcal{H}^n(x)
&=\lim_{t\rightarrow+\infty}\int_{z+C}\chi_{z+\frac{1}{t}\mathbb{A}}(x)\langle\nabla\Theta(x),z\rangle\,d\mathcal{H}^n(x)\\
&=\int_{z+C}\langle\nabla\Theta(x),z\rangle\lim_{t\rightarrow+\infty}\chi_{z+\frac{1}{t}\mathbb{A}}(x)\,d\mathcal{H}^n(x)\\
&=\int_{z+C}\langle\nabla\Theta(x),z\rangle\chi_{\cup_{t>0}(z+\frac{1}{t}\mathbb{A})}(x)\,d\mathcal{H}^n(x)\\
&=\int_{z+C}\langle\nabla\Theta(x),z\rangle\chi_{(z+C)\setminus\{o\}}(x)\,d\mathcal{H}^n(x)\\
&=\int_{z+C}\langle\nabla\Theta(x),z\rangle\,d\mathcal{H}^n(x).
\end{aligned}
\end{equation*}
Let $f(t)=T_\Theta(\mathbb{A},tz)$. By the L'Hopital's rule, Lemma \ref{Decay estimate of the asymptotic weighted co-volume functinal}, and formula \eqref{Deriative-fornula-of-asymptotic-co-volume}, we have
\begin{equation*}
\begin{aligned}
0&=\lim_{t\rightarrow+\infty}\frac{f(t)}{t^{n-q}}\\
&=\lim_{t\rightarrow+\infty}\frac{f'(t)}{(n-q)t^{n-q-1}}\\
&=\frac{1}{n-q}\lim_{t\rightarrow+\infty}t^{q+1-n}\frac{d\,T_\Theta(\mathbb{A},tz)}{dt}\\
&=\boldsymbol{\chi}_{-C}\ast\Theta\,(z)-\frac{1}{n-q}\lim_{t\rightarrow+\infty}
\int_{z+\frac{1}{t}\mathbb{A}}\langle\nabla\Theta(x),z\rangle\,d\mathcal{H}^n(x)\\
&=\boldsymbol{\chi}_{-C}\ast\Theta\,(z)-\frac{1}{n-q}\int_{z+C}\langle\nabla\Theta(x),z\rangle\,d\mathcal{H}^n(x),
\end{aligned}
\end{equation*}
which gives
\begin{equation*}
\begin{aligned}
\boldsymbol{\chi}_{-C}\ast\Theta\,(z)&=\frac{1}{n-q}\int_{z+C}\langle\nabla\Theta(x),z\rangle\,d\mathcal{H}^n(x)\\
&=\frac{1}{n-q}\int_{z+C}\frac{\partial\Theta}{\partial z}(x)\,d\mathcal{H}^n(x).
\end{aligned}
\end{equation*}
This completes the proof.
\end{proof}
\begin{remark}
Noting that the directional derivative of $\Theta$ along $z$ may be positive or negative, as in the case of $\Theta(x)=|x|^{-q}$, the result in Theorem \ref{Formula-convolution-and-driection-derivative-of-Theta} indicates that its integral over $z+C$ must be negative.
\end{remark}

\section{The Brunn-Minkowski type inequality for $C$-pseudo-cones}\label{Section-ABM-inequality-and-conjecture}

Let $\mathbb{A}_1$ and $\mathbb{A}_2$ be two $C$-close sets. The co-sum of two $C$-coconvex sets $A_1=C\setminus\mathbb{A}_1$ and $A_2=C\setminus\mathbb{A}_2$ is defined by $A_1\oplus A_2=C\setminus(\mathbb{A}_1+\mathbb{A}_2)$, where $``+"$ is the usual Minkowski sum. The closeness of the co-sum operator is guaranteed by the complemented Brunn-Minkowski inequality for $C$-coconvex sets, which has been studied by Khovanski\u{i} and Timorin \cite{Khovanskii-Timorin-On_the_theory_of_coconvex_bodies}, as well as Milman and Rotem \cite{Milman-Rotem-Complemented_Brunn_Minkowski_inequalities}. Schneider \cite{Schneider-A_Brunn_Minkowski_theory} established the equality condition for this inequality. Additionally, Milman and Rotem \cite[p.~895-902]{Milman-Rotem-Complemented_Brunn_Minkowski_inequalities} also explored the complemented Brunn-Minkowski inequality for star-shaped sets and radial sum.
Inspired by their work, we aims to consider the asymptotic Brunn-Minkowski inequality for $C$-pseudo-cones in this section.

For convenience, we define the radial sum $E_1\widetilde{+}E_2$ of $C$-pseudo-cones $E_1$ and $E_2$ as follows:
\begin{equation*}
E_1\widetilde{+}E_2\triangleq\text{cl}\,\{\lambda u\,|\,\lambda\geqslant\rho_{E_1\widetilde{+}E_2}(u)
=\rho_{E_1}(u)+\rho_{E_2}(u),\,u\in\Omega_C\},
\end{equation*}
where $E_1\widetilde{+}E_2$ is generally not a $C$-pseudo-cone. The radial function of a $C$-coconvex set $A$ is defined by $\rho_A(u)\triangleq\rho_{C\setminus A}(u),\,u\in\Omega_C$, which is a locally Lipschitz continuous function on $\Omega_C$ by Lemma \ref{locally-Lipschitz-radial-map-of-C-pseudo-cone}. The radial sum $A_1\widetilde{+}A_2$ of $C$-coconvex sets $A_1$ and $A_2$ is defined by
\begin{equation*}
A_1\widetilde{+}A_2\triangleq C\setminus\big((C\setminus A_1)\widetilde{+}(C\setminus A_2)\big).
\end{equation*}
Regarding the radial sum and the co-sum, we have the following result:
\begin{lemma}\label{The-radial-sum-and-co-sum-relationship}
Let $E_1$ and $E_2$ be two $C$-pseudo-cones. Then
\begin{equation*}
E_1\widetilde{+}E_2\subset E_1+E_2.
\end{equation*}
In particular, for $C$-coconvex sets $A_1,A_2$, we have
\begin{equation*}
A_1\oplus A_2\subset A_1\widetilde{+}A_2.
\end{equation*}
\end{lemma}
\begin{proof}
For each $u\in\Omega_C$, we have $\rho_{E_1}(u)u\in E_1$ and $\rho_{E_2}(u)u\in E_2$, then
\begin{equation*}
\rho_{E_1\widetilde{+}E_2}(u)u=\rho_{E_1}(u)u+\rho_{E_2}(u)u\in E_1+E_2.
\end{equation*}
Since $E_1+E_2$ is a $C$-pseudo-cone by Corollary \ref{Minkowski-sum-of-C-pseudo-cone}, for any $\lambda\geqslant\rho_{E_1\widetilde{+}E_2}(u)$, we have
\begin{equation*}
\lambda u\in E_1+E_2.
\end{equation*}
Thus, we can conclude that
\begin{equation*}
\{\lambda u\,|\,\lambda\geqslant\rho_{E_1\widetilde{+}E_2}(u),\,u\in\Omega_C\}\subset E_1+E_2,
\end{equation*}
which implies
\begin{equation*}
E_1\widetilde{+}E_2=\text{cl}\,\{\lambda u\,|\,\lambda\geqslant\rho_{E_1\widetilde{+}E_2}(u),\,u\in\Omega_C\}
\subset\text{cl}\,(E_1+E_2)=E_1+E_2.
\end{equation*}
For $C$-coconvex sets $A_1,A_2$, we have
\begin{equation*}
A_1\oplus A_2=C\setminus(C\setminus A_1+C\setminus A_2)\subset C\setminus((C\setminus A_1)\widetilde{+}(C\setminus A_2))
=A_1\widetilde{+}A_2.
\end{equation*}
\end{proof}

Suppose that $0\leqslant q\leqslant n-1$, we call a $C$-pseudo-cone $E$ a $(C,\Theta)$-close set if
\begin{equation*}
\overline{V}_\Theta(E)=\int_{C\setminus E}\Theta(x)\,dx<+\infty.
\end{equation*}
If $q=0$ and $\Theta\equiv1$, then a $(C,\Theta)$-close set is just a $C$-close set.

\begin{proof}[Proof of Theorem \ref{Theorem-Brunn-Minkowski-inequality-q-small-than-n-1}]
By Lemma \ref{The-radial-sum-and-co-sum-relationship}, we have
\begin{equation*}
\overline{V}_\Theta(E_1+E_2)=V_\Theta(C\setminus(E_1+E_2))\leqslant V_\Theta(C\setminus(E_1\widetilde{+}E_2))
\triangleq\overline{V}_\Theta(E_1\widetilde{+}E_2),
\end{equation*}
with equality if and only if $E_1+E_2=E_1\widetilde{+}E_2$ by the properties of continuous functions. Using the polar coordinates formula, we have
\begin{equation*}
\begin{aligned}
\overline{V}_\Theta(E_1\widetilde{+}E_2)&=\int_{C\setminus(E_1\widetilde{+}E_2)}\Theta(x)\,d\mathcal{H}^n(x)\\
&=\int_{\Omega_C}\int_0^{\rho_{E_1\widetilde{+}E_2}(u)}\Theta(ru)r^{n-1}\,drdu\\
&=\frac{1}{n-q}\int_{\Omega_C}\Theta(u)\rho^{n-q}_{E_1\widetilde{+}E_2}(u)\,du\\
&=\frac{1}{n-q}\int_{\Omega_C}\Theta(u)\big(\rho_{E_1}(u)+\rho_{E_2}(u)\big)^{n-q}\,du.
\end{aligned}
\end{equation*}
Since $n-q\in[1,n]$, by the Minkowski inequality for the $p$-norm, we obtain
\begin{equation*}
\begin{aligned}
\overline{V}_\Theta(E_1\widetilde{+}E_2)^\frac{1}{n-q}&=\bigg(\frac{1}{n-q}\int_{\Omega_C}\Theta(u)
\big(\rho_{E_1}(u)+\rho_{E_2}(u)\big)^{n-q}\,du\bigg)^\frac{1}{n-q}\\
&\leqslant\bigg(\frac{1}{n-q}\int_{\Omega_C}\Theta(u)\rho^{n-q}_{E_1}(u)\,du\bigg)^\frac{1}{n-q}
+\bigg(\frac{1}{n-q}\int_{\Omega_C}\Theta(u)\rho^{n-q}_{E_2}(u)\,du\bigg)^\frac{1}{n-q}\\
&=\overline{V}_\Theta(E_1)^\frac{1}{n-q}+\overline{V}_\Theta(E_2)^\frac{1}{n-q}.
\end{aligned}
\end{equation*}
Equality holds if and only if $E_1$ and $E_2$ are dilates of each other. Note that if $E_1$ and $E_2$ are dilates of each other, then $E_1+E_2=E_1\widetilde{+}E_2$. Therefore, we have
\begin{equation*}
\overline{V}_\Theta(E_1+E_2)^\frac{1}{n-q}\leqslant\overline{V}_\Theta(E_1\widetilde{+}E_2)^\frac{1}{n-q}
\leqslant\overline{V}_\Theta(E_1)^\frac{1}{n-q}+\overline{V}_\Theta(E_2)^\frac{1}{n-q},
\end{equation*}
where the two equalities hold if and only if $E_1$ and $E_2$ are dilates of each other.
\end{proof}

When $q=0$ and $\Theta\equiv1$, we obtain the following corollary:
\begin{corollary}
Let $A_1,A_2$ be two $C$-coconvex sets. Then
\begin{equation*}
V(A_1\oplus A_2)^{\frac{1}{n}}\leqslant V(A_1\widetilde{+}A_2)^{\frac{1}{n}}\leqslant V(A_1)^{\frac{1}{n}}+V(A_2)^{\frac{1}{n}}.
\end{equation*}
Both equalities hold if and only if $E_1$ and $E_2$ are dilates of each other.
\end{corollary}
This corollary strengthens the complementary Brunn-Minkowski inequality from \cite{Schneider-A_Brunn_Minkowski_theory} and identifies its equality condition through a different approach. It is important to note that the key to establishing the equality condition is the closeness of the Minkowski sum of $C$-pseudo-cones, which is guaranteed by our asymptotic theory and is completely independent of the methods used in \cite{Schneider-A_Brunn_Minkowski_theory}.

The case of $q\in[0,n]$ is satisfactory. However, for $n-1<q<n$, Schneider \cite{Schneider-A_weighted_Minkowski_theorem} noted that the solution to the weighted Minkowski problem lacks uniqueness, which implies that Brunn-Minkowski type inequalities may not exist in this case. In fact, according to \cite[Corollary 4.4]{Milman-Rotem-Complemented_Brunn_Minkowski_inequalities}, a Borel measure $\mu$ with a homogeneous density $w:\mathbb{R}^n\setminus\{o\}\rightarrow\mathbb{R}_+$ of degree $\frac{1}{p}$ satisfies the $s$-complemented Brunn-Minkowski inequality
\begin{equation*}
\mu[\mathbb{R}^n\setminus\big(\lambda A+(1-\lambda)B\big)]\leqslant[\lambda\mu(\mathbb{R}^n\setminus A)^s
+(1-\lambda)\mu(\mathbb{R}^n\setminus B)^s]^{\frac{1}{s}},
\end{equation*}
where $p\in(-\infty,-\frac{1}{n-1}]\cup(-\frac{1}{n},0)\cup(0,+\infty]$ and
\begin{equation*}
\frac{1}{s}=\frac{1}{p}+n.
\end{equation*}
Let $\frac{1}{p}=-q$, then $s=\frac{1}{n-q}$ and hence $q\in(-\infty,n-1]\cup(n,+\infty)$. Thus, it is evident that Brunn-Minkowski type inequality may not hold in this case. Moreover, for $q\geqslant n$, the weighted co-volume functional is infinite for any $C$-pseudo-cone, as established in Theorem \ref{Theorem-finite-table}. However, utilizing our asymptotic theory, we conjecture that the following Brunn-Minkowski type inequality holds. We also believe that the convolution formula in Theorem \ref{Formula-convolution-and-driection-derivative-of-Theta} will be instrumental in establishing the inequality.
\vskip 2mm
\noindent {\bf Asymptotic Brunn-Minkowski inequality for $C$-pseudo-cones:} Let $E_1,E_2$ be two $C$-pseudo-cones with their $C$-starting points $z_1,z_2\neq o$, and let $\lambda\in(0,1)$. If $n\neq q\geqslant 0$, then
\begin{equation*}
T_\Theta(\lambda E_1+(1-\lambda)E_2)^{\frac{1}{n-q}}\leqslant\lambda T_\Theta(E_1)^{\frac{1}{n-q}}
+(1-\lambda)T_\Theta(E_2)^{\frac{1}{n-q}}.
\end{equation*}

\section{The weighted Minkowski problem for $q\in[0,n-1)$}

As an application of our asymptotic theory, we will consider the solutions to the weighted Minkowski problem in this section. First, we recall some notations from \cite{Schneider-A_weighted_Minkowski_theorem}. Let $b(E)$ denote the distance of the $C$-pseudo-cone $E$ from the origin, and let $\delta_C(u)$ represent the spherical distance of $u$ from $\partial\Omega_{C^\circ}$ for $u\in\Omega_{C^\circ}$. For $\alpha>0$, we define $\omega(\alpha)=\{u\in\Omega_{C^\circ}\,|\,\delta_C(u)\geqslant\alpha\}$.
According to \cite{Schneider-A_Brunn_Minkowski_theory}, a set $K$ is \emph{$C$-determined} by the compact set $\emptyset\neq\omega\subset\Omega_{C^\circ}$ if
\begin{equation*}
K=C\cap\bigcap_{u\in\omega}H^-(u,h_K(u)).
\end{equation*}
Denote by $\mathcal{K}(C,\omega)$ the set of $C$-pseudo-cones that are $C$-determined by $\omega$. Let $h:\omega\to\mathbb{R}$ be a positive continuous function. The Wulff shape associated with $(C,\omega,h)$ is defined by
\begin{equation*}
[h]=C\cap\bigcap_{v\in\omega}\{x\in\mathbb{R}^n\,|\,\langle x,v\rangle\leqslant-h(v)\},
\end{equation*}
which belongs to $\mathcal{K}(C,\omega)$ (See \cite{Li-Ye-Zhu-The_dual_Minkowski_problem,Schneider-A_Brunn_Minkowski_theory,Schneider-A_weighted_Minkowski_theorem} for more details).
For $0\leqslant q< n-1$, using methods similar to those in \cite{Schneider-A_weighted_Minkowski_theorem}, we have the following Lemmas \ref{lemma-variational-formula}-\ref{Lower bound estimate}:
\begin{lemma}[see \cite{Schneider-A_weighted_Minkowski_theorem}]\label{lemma-variational-formula}
Let $\omega\subset\Omega_{C^\circ}$ be a nonempty compact set, and let $K\in\mathcal{K}(C,\omega)$. Let $f:\omega\rightarrow\mathbb{R}$ be continuous, and let $\big[\overline h_K|_\omega+tf\big]$ be the Wulff shape associated with $(C,\omega, \overline h_K|_\omega+tf)$. Then
\begin{equation}\label{Variational formula}
\lim_{t\rightarrow0}\frac{\overline{V}_\Theta(\big[\overline h_K|_\omega+tf\big])-\overline{V}_\Theta(K)}{t}
=\int_\omega f(u)\,dS_{n-1}^\Theta(K,u).
\end{equation}
\end{lemma}
\begin{lemma}[see \cite{Schneider-A_weighted_Minkowski_theorem}, Upper bound estimate]\label{Upper bound estimate}
There exists a constant $\Lambda$, depending only on $C$ and $\Theta$, such that every $(C,\Theta)$-close set $E$ with $\overline{V}_\Theta(E)=1$ satisfies $\overline{h}_E\leqslant\Lambda$.
\end{lemma}
\begin{lemma}[see \cite{Schneider-A_weighted_Minkowski_theorem}, Lower bound estimate]\label{Lower bound estimate}
There exsts a number $t>0$ such that
\begin{equation*}
\{K\in\mathcal{K}(C,\omega)\}\wedge\{\overline{V}_\Theta(K)=1\}\Rightarrow C(t)\subset K.
\end{equation*}
\end{lemma}

\begin{proof}[Proof of Theorem \ref{Weighted-Minkowski-problem-for-q-small-than-n-1}]
The existence of a solution to the weight Minkowski problem will be divided into two steps as follows.

\noindent\emph{Step 1}: Assume that $\omega\subset\Omega_{C^\circ}$ is compact and that $\mu$ is a nonzero finite Borel measure on $\omega$. We consider the functional $\mathcal{F}:\mathbb{C}^+(\omega)\rightarrow(0,\infty)$ defined by
\begin{equation}
\mathcal{F}(f)=\overline{V}_\Theta([f])^{-\frac{1}{n-q}}\int_\omega f\,d\mu,
\end{equation}
where $[f]$ is the Wulff shape of $f\in\mathbb{C}^+(\omega)$. Suppose there exists an $f_0\in\mathbb{C}^+(\omega)$ such that
\begin{equation*}
\mathcal{F}(f_0)=\sup_{f\in\mathbb{C}^+(\omega)}\mathcal{F}(f),
\end{equation*}
where $f_0=\overline{h}_{[f_0]}$ due to $f\leqslant\overline{h}_{[f]}$ for every $f\in\mathbb{C}^+(\omega)$. By the variational formula \eqref{Variational formula}, for each $f\in\mathbb{C}^+(\omega)$, we have
\begin{equation*}
\begin{aligned}
0&=\frac{\partial\mathcal{F}(f_0+tf)}{\partial t}\bigg|_{t=0}\\
&=-\frac{1}{n-q}\overline{V}_\Theta([f_0])^{-\frac{1}{n-q}-1}\int_\omega f(u)\,dS_{n-1}^\Theta([f_0],u)\int_\omega f_0\,d\mu
+\overline{V}_\Theta([f_0])^{-\frac{1}{n-q}}\int_\omega f\,d\mu\\
&=\overline{V}_\Theta([f_0])^{-\frac{1}{n-q}}\bigg(\int_\omega f\,d\mu-\frac{1}{(n-q)\overline{V}_\Theta([f_0])}\int_\omega f_0\,d\mu
\int_\omega f(u)\,dS_{n-1}^\Theta([f_0],u)\bigg),
\end{aligned}
\end{equation*}
where $0<\overline{V}_\Theta([f_0])<+\infty$ due to $q<n$ and $[f_0]\in\mathcal{K}(C,\omega)$. Thus, the Euler-Lagrange equation of the functional $\mathcal{F}$ is
\begin{equation}\label{the Euler-Lagrange equation}
\int_\omega f\,d\mu=\int_\omega f(u)\,dS_{n-1}^\Theta(\widetilde{[f_0]},u),
\end{equation}
where
\begin{equation*}
\widetilde{[f_0]}=\bigg(\frac{1}{(n-q)\overline{V}_\Theta([f_0])}\int_\omega f_0\,d\mu\bigg)^\frac{1}{n-1-q}[f_0].
\end{equation*}
Applying the Riesz representation theorem to the Euler-Lagrange equation \eqref{the Euler-Lagrange equation}, we have
\begin{equation*}
S_{n-1}^\Theta(\widetilde{[f_0]},\cdot)=\mu.
\end{equation*}
Therefore, we need to show that the functional $\mathcal{F}$ has maximum. Similar to Lemmas 7 and 8 in \cite{Schneider-A_weighted_Minkowski_theorem}, for $0\leqslant q<n-1$, the functional $\mathcal{F}$ remains a $0$-homogeneous continuous functional, so we have
\begin{equation*}
\sup_{f\in\mathbb{C}^+(\omega)}\mathcal{F}(f)=\sup_{f\in\mathcal{L}}\mathcal{F}(f),
\end{equation*}
where $\mathcal{L}=\{\overline{h}_K|_\omega\,|\,K\in\mathcal{K}(C,\omega),\,\overline{V}_\Theta(K)=1\}$. For any $\overline{h}_K|_\omega\in\mathcal{L}$, since $K$ is a $(C,\Theta)$-close set, by Lemma \ref{Upper bound estimate}, we have $\mathcal{F}(\overline{h}_K|_\omega)\leqslant\Lambda$ for some constant $\Lambda>0$. Combining Lemma \ref{Lower bound estimate} and Schneider's selection theorem (see \cite[Lemma 1]{Schneider-Pseudo_cones}), we conclude that $\mathcal{F}$ attains a maximum on $\mathcal{L}$.

\noindent\emph{Step 2}: Now we assume that $\mu$ is a nonzero finite Borel measure on $\Omega_{C^\circ}$. Choose a number $\tau>0$ such that $\mu(\omega(\tau))>0$ and a sequence $(\omega_j)_{j=1}^{+\infty}$ of compact subsets of $\Omega_{C^\circ}$ such that
\begin{equation*}
\omega_1=\omega(\tau),\,\omega_j\subset\text{int}\,\omega_{j+1},\,\bigcup_{j=1}^{+\infty}\omega_j=\Omega_{C^\circ}.
\end{equation*}
Define measures $\mu_j$ by $\mu_j(\omega)=\mu(\omega\cap\omega_j)$ for each $j\in\mathbb{N}$ and every $\omega\in\mathcal{B}(\Omega_{C^\circ})$. From step 1, there exists a sequence of $K_j\in\mathcal{K}(C,\omega_j)$ such that
\begin{equation*}
\lambda_jS^\Theta_{n-1}(K_j,\cdot)=\mu_j, \ \,\overline{V}_\Theta(K_j)=1
\end{equation*}
with
\begin{equation*}
\lambda_j=\frac{1}{n-q}\int_{\omega_j}\overline{h}_{K_j}\,d\mu_j=\frac{1}{n-q}\int_{\omega_j}\overline{h}_{K_j}\,d\mu.
\end{equation*}
Let $L_j=\lambda_j^\frac{1}{n-1-q}K_j$, then $S^\Theta_{n-1}(L_j,\cdot)=\mu_j$. To apply the Blaschke selection theorem for the sequence $\{L_j\}_{j=1}^{+\infty}$, we need upper and lower bound estimates for the distances $b(L_j)$. By Lemma \ref{Upper bound estimate}, we have $\overline h_{K_j}\leqslant\Lambda$ with a constant $\Lambda$ independent of $j$. Since
\begin{equation*}
\lambda_j=\frac{1}{n-q}\int_{\omega_j}\overline{h}_{K_j}\,d\mu\leqslant\frac{1}{n-q}\int_{\Omega_{C^\circ}}\Lambda\,d\mu
=\frac{\Lambda}{n-q}\mu(\Omega_{C^\circ})<\infty,
\end{equation*}
it follows that
\begin{equation*}
\overline{h}_{L_j}=\lambda_j^\frac{1}{n-1-q}\overline{h}_{K_j}\leqslant
\Big(\frac{\Lambda}{n-q}\mu(\Omega_{C^\circ})\Big)^\frac{1}{n-1-q}\Lambda.
\end{equation*}
Thus, $\{b(L_j)\}_{j=1}^{+\infty}$ is bounded from above. Choose a number $t_1>0$ such that
\begin{equation*}
L_j\cap C^-(t_1)\not=\emptyset, \ \forall\,j\in\mathbb{N},
\end{equation*}
and we have $S_{n-1}(L_j,\omega(\tau))=\varphi(\omega(\tau))\triangleq s>0$. Then, by \cite[Lemma 9]{Schneider-Pseudo_cones}, we conclude that $b(L_j)>\lambda$ for some constant $\lambda>0$ depending only on $C,\tau$ and $s$. Therefore, using the Blaschke selection theorem, there exists a $C$-pseudo-cone $K$ such that $L_j\rightarrow K$ as $j\rightarrow+\infty$. Following the methods in \cite{Schneider-A_weighted_Minkowski_theorem}, we find that
\begin{equation*}
S^\Theta_{n-1}(K,\cdot)=\mu.
\end{equation*}
Using the formula \eqref{Integral-formula-weighted-surface-area-measure-weighted-co-volume}, we have
\begin{equation*}
\begin{aligned}
\overline{V}_{\Theta}(L_j)&=\frac{1}{n-q}\int_{\Omega_{C^\circ}}\overline{h}_{L_j}(v)\,dS^{\Theta}_{n-1}(L_j,v)\\
&=\frac{1}{n-q}\int_{\Omega_{C^\circ}}\overline{h}_{L_j}(v)\,d\mu_j(v)\\
&\leqslant\Big(\frac{\Lambda}{n-q}\mu(\Omega_{C^\circ})\Big)^\frac{1}{n-1-q}\Lambda\mu(\Omega_{C^\circ})\triangleq\alpha(\Lambda,C,q).
\end{aligned}
\end{equation*}
By the continuity of the weighted co-volume, we have $\overline{V}_{\Theta}(K)\leqslant\alpha(\Lambda,C,q)$, which means $K$ is a $(C,\Theta)$-close set.

Next, we prove the uniqueness. Let $K,L\in\mathcal{K}(C,\omega)$ satisfy
\begin{equation*}
S^\Theta_{n-1}(K,\cdot)=S^\Theta_{n-1}(L,\cdot)=\mu.
\end{equation*}
By the weighted Brunn-Minkowski inequality \eqref{Formula-Brunn-Minkowski-inequality-q-small-than-n-1}, the function
\begin{equation*}
f(t)=\overline{V}_\Theta(K+tL)^\frac{1}{n-q}-\overline{V}_\Theta(K)^\frac{1}{n-q}-t\overline{V}_\Theta(L)^\frac{1}{n-q}
\end{equation*}
is a negative convex function, which implies $f'(0)\leqslant0$. Applying the variational formula \eqref{Variational formula}, we have
\begin{equation*}
\begin{aligned}
f'(0)&=\frac{1}{n-q}\overline{V}_\Theta(K)^{\frac{1}{n-q}-1}\int_\omega\overline{h}_L(u)\,dS_{n-1}^\Theta(K,u)
-\overline{V}_\Theta(L)^\frac{1}{n-q}\\
&=\frac{1}{n-q}\overline{V}_\Theta(K)^{\frac{1}{n-q}-1}\int_\omega\overline{h}_L(u)\,dS_{n-1}^\Theta(L,u)
-\overline{V}_\Theta(L)^\frac{1}{n-q}\\
&=\overline{V}_\Theta(K)^{\frac{1}{n-q}-1}\overline{V}_\Theta(L)-\overline{V}_\Theta(L)^\frac{1}{n-q}.
\end{aligned}
\end{equation*}
Thus, we obtain
\begin{equation*}
\overline{V}_\Theta(K)\leqslant\overline{V}_\Theta(L).
\end{equation*}
By switching the roles of $K$ and $L$, we also find that
$$
\overline{V}_\Theta(L)\leqslant\overline{V}_\Theta(K).
$$
This shows that
\begin{equation*}
\overline{V}_\Theta(K)=\overline{V}_\Theta(L).
\end{equation*}
Consequently, we have
\begin{equation*}
f'(0)=\overline{V}_\Theta(K)^{\frac{1}{n-q}-1}\overline{V}_\Theta(L)-\overline{V}_\Theta(L)^\frac{1}{n-q}=0,
\end{equation*}
which shows that the negative convex function $f\equiv0$. By the equality condition of the inequality \eqref{Formula-Brunn-Minkowski-inequality-q-small-than-n-1} and the $(n-1-q)$-homogeneity of the weighted surface area measure, we conclude that $K=L$.
\end{proof}

\vskip 10mm

\noindent Xudong Wang, \ \ {\small \tt xdwang@snnu.edu.cn}\\
{\em School of Mathematics and Statistics, Shaanxi Normal University\\
Xi'an, 710119, China}

\vskip 2mm \noindent Wenxue Xu, \ \ {\small \tt xuwenxue83@swu.edu.cn}\\
{\em School of Mathematics and Statistics, Southwest University\\
Chongqing, 400715, China}

\vskip 2mm \noindent Jiazu Zhou, \ \ {\small \tt zhoujz@swu.edu.cn}\\
{\em School of Mathematics and Statistics, Southwest University\\
Chongqing, 400715, China}

\vskip 2mm \noindent Baocheng Zhu, \ \ {\small \tt bczhu@snnu.edu.cn}\\
{\em School of Mathematics and Statistics, Shaanxi Normal University\\
Xi'an, 710119, China}

\end{document}